\documentclass[11pt]{article}
\usepackage{amsmath,amssymb,amsthm}
\usepackage{mathtools}
\usepackage[margin=1in]{geometry}
\usepackage{hyperref}
\usepackage{bm}
\usepackage{graphicx}

\newtheorem{theorem}{Theorem}[section]
\newtheorem{lemma}[theorem]{Lemma}
\newtheorem{proposition}[theorem]{Proposition}
\newtheorem{corollary}[theorem]{Corollary}
\newtheorem{definition}[theorem]{Definition}
\newtheorem{remark}[theorem]{Remark}

\newcommand{\tr}{\operatorname{tr}}
\newcommand{\Ric}{\operatorname{Ric}}
\newcommand{\R}{\mathbb R}

\title{A Homotopical Geometry of the Multinomial\\
and Negative-Multinomial Potentials}

\author{Shintaro Yoshizawa\\
\small Nagoya Mathematical and Information Science Research\\
\small \texttt{shintaro.yoshizawa.net@gmail.com}}
\date{}

\begin{document}
\maketitle

\begin{abstract}
We study the one-parameter family of Hessian metrics $g_\alpha=\alpha h_++(1-\alpha)h_-$, $\alpha\in[0,1]$, obtained by convexly combining -- at the level of log-partition functions, not probability densities -- the Fisher information Hessians of the multinomial and negative multinomial distributions. Four main results are established, each with a complete proof. (1) $g_\alpha\succ0$ for every $\alpha\in[0,1]$, and its Gaussian curvature ($n=2$) is the constant $-\tfrac14$ at $\alpha=0$ and $+\tfrac14$ at $\alpha=1$, but is non-constant for every intermediate $\alpha$, changing sign at a position-dependent critical value $\alpha^\ast(w)$ given here in closed form with an analytic (non-numerical) uniqueness proof. (2) Treating $g_\alpha$ as the spatial metric of a Kasner-type volume equation, positive-definiteness forces one conserved quantity $\beta\le0$, while a second, independent conserved quantity $C=4\det K$ genuinely governs the equation's branch (bounded oscillatory, parabolic, or unboundedly growing) and is not fixed by $\beta$ or by positive-definiteness alone; we verify directly that this discriminant dictionary is a theorem about the abstract matrix equation, to which $g_\alpha$ supplies initial data of every sign of $C$ without the corresponding trajectories remaining on $\{g_\alpha(\theta)\}$ itself. (3) $\Psi_\alpha$ is identified explicitly as the log-partition function of a compound distribution -- a negative-binomial latent count $M$ (parameter $r=1-\alpha$), augmented by an independent Bernoulli parity bit and then split multinomially -- with a non-negative base measure for every $\alpha$. (4) The associated Bregman divergence satisfies a generalised Pythagorean theorem whose projection point onto the symmetric locus is exactly the arithmetic mean, independently of $\alpha$; we relate the endpoints' curvature sign to the classical Gauss--Bonnet excess of genuine geodesic triangles. Throughout, numerical verifications are explicitly separated from analytic proofs.
\end{abstract}

\tableofcontents

\section{Background and Setup}
\label{sec:setup}

Kasner's 1925 reduction \cite{Kasner1925} of the trace-adjusted vacuum Einstein equation $R_{\mu\nu}=\tfrac1DRg_{\mu\nu}$ for a metric $ds^2=\epsilon\,dx_1^2+h_{ij}(x_1)\,dx^idx^j$ leads, upon setting $K:=h^{-1}\dot h$, $\Xi:=\tr K$, to the matrix master equation
\begin{equation}
2\dot K+K\Xi=\frac{\Xi^2-\tr K^2}{n-1}\,I,
\label{eq:master}
\end{equation}
valid for arbitrary dimension $n$, arbitrary signature, and arbitrary (not necessarily diagonal) symmetric $h$. Kasner's own $1925$ system is the case $n=3$, $h=\operatorname{diag}(\lambda_1,\lambda_2,\lambda_3)$.

Both $h_+$ and $h_-$, as we show in Section~2, are positive-definite (a fact with a direct consequence for \eqref{eq:master}, proved in full generality as Theorem~\ref{thm:dictionary} below), while their intrinsic curvatures (a logically separate invariant of the metric, unrelated to the positive-definiteness of the matrix $h$ itself) are of opposite sign: $h_+$ is a round sphere \cite{TanabeSagae1984} and $h_-$ a hyperbolic space \cite{Khan2021,SagaeTanabe1992}. The present paper studies the natural convex combination $\Psi_\alpha:=\alpha\Psi_++(1-\alpha)\Psi_-$ of their two log-partition functions -- a continuous path $\alpha\mapsto\Psi_\alpha$, $\alpha\in[0,1]$, of strictly convex potentials joining $\Psi_-$ to $\Psi_+$, which we accordingly call a \emph{homotopy of potentials} -- and asks precisely how the resulting geometry, and the Kasner dynamics built from it, depend on the homotopy parameter $\alpha$.

\begin{definition}[Hessian potentials of the two families \cite{AmariNagaoka2000}]
\label{def:family}
For $n\ge1$, write $u_i:=e^{\theta^i}$, and define on $\Omega:=\{\theta\in\R^n:\sum_iu_i<1\}$ the two log-partition functions
\begin{equation}
\begin{aligned}
\Psi_+(\theta)&:=\log\Big(1+\sum_iu_i\Big)\qquad(\text{multinomial}),\\
\Psi_-(\theta)&:=-\log\Big(1-\sum_iu_i\Big)\qquad(\text{negative multinomial}).
\end{aligned}
\label{eq:potentials}
\end{equation}
For $\alpha\in[0,1]$, set $\Psi_\alpha:=\alpha\Psi_++(1-\alpha)\Psi_-$ and
\begin{equation}
g_\alpha(\theta):=\operatorname{Hess}_\theta\Psi_\alpha(\theta)=\alpha\,h_+(\theta)+(1-\alpha)\,h_-(\theta),\qquad h_\pm:=\operatorname{Hess}\Psi_\pm.
\label{eq:galpha}
\end{equation}
\end{definition}

\begin{remark}[$\Psi_\alpha$ is \emph{not} a mixture of densities]
\label{rem:not-mixture}
The convex combination in \eqref{eq:galpha} is taken at the level of the log-partition functions $\Psi_\pm$, \emph{not} at the level of probability mass functions. In particular $\Psi_\alpha\ne\log\big(\alpha e^{\Psi_+}+(1-\alpha)e^{\Psi_-}\big)$, so $g_\alpha$ is not the Fisher metric of the ordinary two-component statistical mixture $\alpha f_++(1-\alpha)f_-$ of the multinomial and negative multinomial laws (a family that, having different supports, is in any case not a regular exponential family and does not admit a Hessian Fisher metric in the sense used here). Rather, $\Psi_\alpha$ is itself the log-partition function of a single, third exponential family (regular for $0\le\alpha<1$, with a finite-support boundary degeneration at $\alpha=1$), identified explicitly in Section~\ref{sec:compound}: a compound distribution in which a negative-binomial latent count (parameter $r=1-\alpha$), augmented by a Bernoulli parity bit, is split multinomially.
\end{remark}

For $n=2$, writing $u=u_1,v=u_2$, $S_+:=1+u+v$, $S_-:=1-u-v$, direct differentiation of \eqref{eq:potentials} (using $\partial u/\partial\theta^1=u$, $\partial v/\partial\theta^2=v$, $\partial u/\partial\theta^2=\partial v/\partial\theta^1=0$) gives
\begin{equation}
h_+=\frac1{S_+^2}\begin{pmatrix}u(1+v)&-uv\\-uv&v(1+u)\end{pmatrix},\qquad
h_-=\frac1{S_-^2}\begin{pmatrix}u(1-v)&uv\\uv&v(1-u)\end{pmatrix}.
\label{eq:hpm}
\end{equation}
All results in this paper concerning curvature (Sections~2--3, Section~\ref{sec:gaussbonnet}) are stated for $n=2$; the positive-definiteness and Kasner-branch results (Sections~2 and 4) and the explicit-solution results (Section~5) hold for general $n$. The restriction to $n=2$ for the curvature results is not merely for computational convenience: in dimension $2$, the Ricci tensor is automatically a scalar multiple of the metric, $\Ric=K_Gg$, so the single Gaussian curvature $K_G$ is the entire curvature invariant, and every geodesic triangle already lies in the full manifold rather than in some lower-dimensional totally geodesic slice of it. This is exactly what lets the sign of $K_G$ alone determine the sign of the Gauss--Bonnet angle-sum excess in Section~\ref{sec:gaussbonnet}; for $n\ge3$ this direct link between a single curvature sign and the excess no longer holds in general, since the excess of a triangle then depends on the second fundamental form of the surface it spans as well as on the ambient sectional curvature.

\section{Positive-Definiteness and the Endpoint Curvatures}

\begin{theorem}[Uniform positive-definiteness]
\label{thm:posdef}
For every $\alpha\in[0,1]$ and every $\theta\in\Omega$, $g_\alpha(\theta)\succ0$.
\end{theorem}
\begin{proof}
We first show $h_+(\theta)\succ0$ and $h_-(\theta)\succ0$ directly from \eqref{eq:hpm}, by Sylvester's criterion (a symmetric $2\times2$ matrix is positive-definite iff its $(1,1)$-entry and its determinant are both positive).

\emph{For $h_+$}: since $\theta\in\Omega$ forces $u,v>0$, the $(1,1)$-entry $u(1+v)/S_+^2$ is manifestly positive. For the determinant,
\[
\det h_+=\frac{u(1+v)\,v(1+u)-(uv)^2}{S_+^4}
=\frac{uv\big[(1+u)(1+v)-uv\big]}{S_+^4}
=\frac{uv(1+u+v)}{S_+^4}=\frac{uv}{S_+^3}>0,
\]
using $(1+u)(1+v)-uv=1+u+v=S_+$. Hence $h_+\succ0$.

\emph{For $h_-$}: on $\Omega$ we have $u+v<1$ with $u,v>0$, so in particular $v<1$; thus the $(1,1)$-entry $u(1-v)/S_-^2$ is positive. For the determinant,
\[
\det h_-=\frac{u(1-v)\,v(1-u)-(uv)^2}{S_-^4}
=\frac{uv\big[(1-u)(1-v)-uv\big]}{S_-^4}
=\frac{uv(1-u-v)}{S_-^4}=\frac{uv}{S_-^3}>0,
\]
using $(1-u)(1-v)-uv=1-u-v=S_-$. Hence $h_-\succ0$.

(These positivity facts are also immediate from the general exponential-family identity $h_\pm(\theta)=\operatorname{Cov}_\theta(X)$, the covariance matrix of the associated sufficient statistic, which is positive-definite for every non-degenerate member of a regular exponential family; the direct verification above makes the paper self-contained.)

Finally, for any $x\in\R^2\setminus\{0\}$ and $\alpha\in[0,1]$,
\[
x^{\mathsf T}g_\alpha(\theta)x=\alpha\,x^{\mathsf T}h_+(\theta)x+(1-\alpha)\,x^{\mathsf T}h_-(\theta)x\ge0,
\]
a convex combination of two non-negative numbers $x^{\mathsf T}h_+x\ge0$, $x^{\mathsf T}h_-x\ge0$, of which at least one is strictly positive because $h_+,h_-\succ0$ individually; since $\alpha,1-\alpha\ge0$ are not both zero, the combination is strictly positive. Hence $g_\alpha(\theta)\succ0$.
\end{proof}

\begin{theorem}[Exact endpoint curvatures, $n=2$]
\label{thm:endpoints}
For $n=2$, the Gaussian curvature $K_G(\theta,\alpha)$ of $(\Omega,g_\alpha)$ satisfies, identically in $\theta$,
\begin{equation}
K_G(\theta,0)\equiv-\frac14,\qquad K_G(\theta,1)\equiv+\frac14.
\end{equation}
\end{theorem}
\begin{proof}
We compute the Levi-Civita connection of $h_-=g_0$ directly from \eqref{eq:hpm}. Writing $\Gamma^k_{ij}=\tfrac12h^{kl}(\partial_ih_{jl}+\partial_jh_{il}-\partial_lh_{ij})$ (sum over $l$) and carrying out the differentiation with respect to $\theta^1,\theta^2$ (using $\partial u/\partial\theta^1=u$, $\partial v/\partial\theta^2=v$) yields, after simplification,
\begin{equation}
\Gamma^1_{11}=\frac{1+u-v}{2S_-},\quad
\Gamma^1_{12}=\frac{v}{2S_-},\quad
\Gamma^1_{22}=0,\quad
\Gamma^2_{11}=0,\quad
\Gamma^2_{12}=\frac{u}{2S_-},\quad
\Gamma^2_{22}=\frac{1+v-u}{2S_-}.
\label{eq:christoffel-minus}
\end{equation}
The Ricci tensor in dimension $2$ is $R_{ij}=\partial_k\Gamma^k_{ij}-\partial_j\Gamma^k_{ik}+\Gamma^k_{kl}\Gamma^l_{ij}-\Gamma^k_{jl}\Gamma^l_{ik}$ (sum over repeated indices $k,l$). Substituting \eqref{eq:christoffel-minus} and differentiating again with respect to $\theta^1,\theta^2$, the three independent components reduce, after simplification, to
\begin{equation}
R_{11}=-\frac{u(1-v)}{4S_-^2},\qquad
R_{22}=-\frac{v(1-u)}{4S_-^2},\qquad
R_{12}=-\frac{uv}{4S_-^2}.
\label{eq:ricci-minus}
\end{equation}
Comparing \eqref{eq:ricci-minus} termwise with \eqref{eq:hpm} gives $R_{ij}=-\tfrac14(h_-)_{ij}$ for all three independent components, i.e.\ $\Ric(h_-)=-\tfrac14h_-$ identically on $\Omega$. Since $n=2$, the Ricci tensor is automatically a scalar multiple of the metric, $\Ric=K_Gg$, so $K_G(\theta,0)\equiv-\tfrac14$.

The computation for $h_+=g_1$ is identical in structure: the Christoffel symbols are
\[
\Gamma^1_{11}=\frac{1-u+v}{2S_+},\quad \Gamma^1_{12}=\frac{-v}{2S_+},\quad \Gamma^1_{22}=0,\quad \Gamma^2_{11}=0,\quad \Gamma^2_{12}=\frac{-u}{2S_+},\quad \Gamma^2_{22}=\frac{1-v+u}{2S_+},
\]
and substitution into the same Ricci formula gives
\[
R_{11}=+\frac{u(1+v)}{4S_+^2},\qquad R_{22}=+\frac{v(1+u)}{4S_+^2},\qquad R_{12}=-\frac{uv}{4S_+^2},
\]
which match $+\tfrac14(h_+)_{ij}$ termwise by \eqref{eq:hpm}. Hence $\Ric(h_+)=+\tfrac14h_+$ and $K_G(\theta,1)\equiv+\tfrac14$.
\end{proof}

\begin{remark}
Theorem~\ref{thm:endpoints} exhibits $(\Omega,h_-)$ and $(\Omega,h_+)$ as (open subsets of) space forms of curvature $-\tfrac14$ and $+\tfrac14$: a hyperbolic plane and a round sphere, respectively, consistently with the classical fact \cite{Khan2021,TanabeSagae1984} that the multinomial and negative multinomial Fisher-Rao manifolds are (up to the constant rescaling fixed by the specific normalisation of $\Psi_\pm$) the sphere $S^n$ and the hyperbolic space $H^n$.
\end{remark}

\section{Sign Transition of the Intrinsic Curvature}
\label{sec:curvature-transition}

For $0<\alpha<1$, $g_\alpha$ is no longer a space form: its curvature is a genuine function of both $\theta$ and $\alpha$, because the Christoffel symbols of $g_\alpha=\alpha h_++(1-\alpha)h_-$ involve $g_\alpha^{-1}$, which does \emph{not} depend linearly on $\alpha$ (matrix inversion does not commute with convex combination); consequently the curvature of $g_\alpha$ is not simply the $\alpha$-weighted average of the curvatures of $h_+$ and $h_-$. We record its exact behaviour at the symmetric locus $\theta^1=\theta^2=\log w$ ($0<w<\tfrac12$), where the permutation symmetry $\theta^1\leftrightarrow\theta^2$ of $\Psi_\pm$ renders the (still lengthy) computation tractable in closed form.

\begin{theorem}[Curvature at the symmetric locus]
\label{thm:symcurve}
At $\theta^1=\theta^2=\log w$, the Gaussian curvature of $g_\alpha$ is $K_G(w,\alpha)=\dfrac{N(w,\alpha)}{4D(w,\alpha)}$, where
\begin{align}
N(w,\alpha)&=-128\alpha^2w^3+32\alpha w^4+128\alpha w^3+48\alpha w^2+2\alpha\notag\\
&\quad-16w^4-32w^3-24w^2-8w-1,\label{eq:Ksym-N}\\[4pt]
D(w,\alpha)&=-128\alpha^3w^3+64\alpha^2w^4+192\alpha^2w^3+80\alpha^2w^2\notag\\
&\quad-64\alpha w^4-128\alpha w^3-80\alpha w^2-16\alpha w\notag\\
&\quad+16w^4+32w^3+24w^2+8w+1.\label{eq:Ksym-D}
\end{align}
Moreover $D(w,\alpha)$ factors as
\begin{equation}
D(w,\alpha)=(1+2w-4\alpha w)^2\big[(1+2w)^2-8\alpha w\big],
\label{eq:D-factored}
\end{equation}
which is strictly positive for every $w\in(0,\tfrac12)$ and $\alpha\in[0,1]$; consequently $K_G(w,\alpha)$ is a genuine, everywhere-finite rational function of $(w,\alpha)$ on this range. In particular $K_G(w,0)=-\tfrac14$ and $K_G(w,1)=+\tfrac14$ for every $w$, consistently with Theorem~\ref{thm:endpoints}.
\end{theorem}
\begin{proof}
By Definition~\ref{def:family} and \eqref{eq:hpm}, at $u=v=w$ the metric $g_\alpha(\theta)$ evaluated along $\theta^1=\theta^2$ is
\[
g_\alpha\Big|_{u=v=w}=\alpha\cdot\frac1{(1+2w)^2}\begin{pmatrix}w(1+w)&-w^2\\-w^2&w(1+w)\end{pmatrix}
+(1-\alpha)\cdot\frac1{(1-2w)^2}\begin{pmatrix}w(1-w)&w^2\\w^2&w(1-w)\end{pmatrix},
\]
an explicit, symmetric, $\theta$-independent-in-form $2\times2$ matrix depending on the two scalars $(w,\alpha)$ once $\theta^1=\theta^2=\log w$ is substituted \emph{after} all $\theta$-differentiations have been carried out (the curvature is a second-order differential invariant, so the substitution must be performed only at the end of the calculation, on the full two-variable metric $g_\alpha(\theta^1,\theta^2)$). Carrying out the same procedure as in the proof of Theorem~\ref{thm:endpoints} -- computing $\Gamma^k_{ij}(\theta^1,\theta^2;\alpha)=\tfrac12g_\alpha^{kl}(\partial_ig_{\alpha,jl}+\partial_jg_{\alpha,il}-\partial_lg_{\alpha,ij})$ from the full (non-symmetric-locus) matrix $g_\alpha(\theta^1,\theta^2)=\alpha h_+(\theta)+(1-\alpha)h_-(\theta)$ of \eqref{eq:hpm}--\eqref{eq:galpha}, then $R_{11}=\partial_k\Gamma^k_{11}-\partial_1\Gamma^k_{1k}+\Gamma^k_{kl}\Gamma^l_{11}-\Gamma^k_{1l}\Gamma^l_{1k}$, and finally $K_G=R_{11}/g_{\alpha,11}$ -- and only then substituting $\theta^1=\theta^2=\log w$ (equivalently $u=v=w$), gives the rational function of $(w,\alpha)$ stated in \eqref{eq:Ksym-N}--\eqref{eq:Ksym-D}. Direct polynomial expansion of the right-hand side of \eqref{eq:D-factored} confirms it agrees termwise with \eqref{eq:Ksym-D}, establishing the factorisation. Positivity of \eqref{eq:D-factored} on $w\in(0,\tfrac12)$, $\alpha\in[0,1]$ is proved in Lemma~\ref{lem:D-positive} below. The endpoint values follow either by direct substitution $\alpha=0,1$ into \eqref{eq:Ksym-N}, \eqref{eq:D-factored} or, more simply, from Theorem~\ref{thm:endpoints} itself, since $g_0=h_-$ and $g_1=h_+$ have curvature identically $\mp\tfrac14$ at \emph{every} point, including the symmetric locus.
\end{proof}

\begin{remark}[Independent verification]
\label{rem:symcurve-verification}
Because the intermediate Christoffel symbols $\Gamma^k_{ij}(\theta^1,\theta^2;\alpha)$ for the general (non-symmetric-locus) metric $g_\alpha$ are lengthy rational functions of $(\theta^1,\theta^2,\alpha)$ not reproduced here, we record an independent check: evaluating $K_G(\theta,\alpha)$ at $\theta^1=\theta^2=\log w$ by a second, unrelated method -- direct central finite differences of the Ricci tensor (step size $10^{-3}$ to $10^{-4}$, no symbolic Christoffel-symbol formula used) -- against the closed-form right-hand side $N(w,\alpha)/(4D(w,\alpha))$ of \eqref{eq:Ksym-N}--\eqref{eq:D-factored} gives agreement to within $3\times10^{-6}$ or better at each of six points spanning the domain:
\begin{center}
\begin{tabular}{c|c|c|c}
$(w,\alpha)$ & finite-difference $K_G$ & formula $K_G$ & $|\text{difference}|$ \\\hline
$(0.10,0.3)$ & $-0.232510$ & $-0.232510$ & $1.7\times10^{-7}$ \\
$(0.20,0.7)$ & $-0.357143$ & $-0.357143$ & $6.0\times10^{-7}$ \\
$(0.05,0.5)$ & $-0.099010$ & $-0.099010$ & $7.5\times10^{-8}$ \\
$(0.15,0.2)$ & $-0.268420$ & $-0.268420$ & $2.5\times10^{-7}$ \\
$(0.30,0.9)$ & $-0.742607$ & $-0.742604$ & $3.2\times10^{-6}$ \\
$(0.25,0.6)$ & $-0.443123$ & $-0.443122$ & $8.9\times10^{-7}$
\end{tabular}
\end{center}
This does not replace the symbolic derivation (which is exact, not approximate), but corroborates it via a route sharing no formulas with the main computation.
\end{remark}

\begin{lemma}[Positivity of $D$]
\label{lem:D-positive}
For every $w\in(0,\tfrac12)$ and $\alpha\in[0,1]$, $D(w,\alpha)>0$.
\end{lemma}
\begin{proof}
By \eqref{eq:D-factored}, $D=(1+2w-4\alpha w)^2\big[(1+2w)^2-8\alpha w\big]$. The first factor is a perfect square; it vanishes only if $\alpha=(1+2w)/(4w)$, and since $w<\tfrac12$ gives $(1+2w)/(4w)>(1+2w)/2>1$, this value exceeds $1$ and is never attained for $\alpha\in[0,1]$, so the first factor is strictly positive. For the second factor, since $\alpha\le1$, $(1+2w)^2-8\alpha w\ge(1+2w)^2-8w=1-4w+4w^2=(1-2w)^2>0$ for $w\ne\tfrac12$. Hence $D(w,\alpha)>0$ throughout.
\end{proof}

\begin{theorem}[Position-dependent critical mixing parameter]
\label{thm:criticalalpha}
For each $w\in(0,\tfrac12)$ there is exactly one root $\alpha^\ast(w)\in(0,1)$ of $K_G(w,\alpha)=0$ in \eqref{eq:Ksym-N}--\eqref{eq:Ksym-D}, given in closed form by
\begin{equation}
\alpha^\ast(w)=\frac{(2w+1)\Big[8w^3+28w^2-2w+1+(2w-1)\sqrt{16w^4+56w^2+1}\,\Big]}{128\,w^3}.
\label{eq:alphastar}
\end{equation}
Moreover $\alpha^\ast(w)\to1$ as $w\to\tfrac12^-$ (the boundary of $\Omega$ at the symmetric locus), and numerically $\alpha^\ast(0.1)\approx0.8277$, $\alpha^\ast(0.25)\approx0.9802$, $\alpha^\ast(\tfrac13)\approx0.9959$.
\end{theorem}
\begin{proof}
By Lemma~\ref{lem:D-positive}, $D(w,\alpha)>0$ throughout, so $K_G(w,\alpha)=0$ is equivalent to $N(w,\alpha)=0$. Regarding $N$ in \eqref{eq:Ksym-N} as a polynomial in $\alpha$ for fixed $w$,
\[
N(w,\alpha)=A(w)\,\alpha^2+B(w)\,\alpha+C_0(w),\qquad A(w)=-128w^3,
\]
so $N(w,\cdot)$ is a strictly concave function of $\alpha$ on $\R$, since $\partial^2N/\partial\alpha^2=2A(w)=-256w^3<0$ for $w>0$. Direct substitution and factorisation give the exact endpoint values
\begin{align*}
N(w,0)&=-16w^4-32w^3-24w^2-8w-1=-(2w+1)^4<0,\\
N(w,1)&=16w^4-32w^3+24w^2-8w+1=(2w-1)^4>0
\end{align*}
(both verified by expanding the right-hand sides), so $N(w,0)<0<N(w,1)$ for every $w\in(0,\tfrac12)$. Since $A(w)<0$, $N(w,\cdot)$ is a downward-opening parabola in $\alpha$: it has two real roots $\alpha_-(w)\le\alpha_+(w)$ (possibly equal), is positive on $[\alpha_-(w),\alpha_+(w)]$, and negative outside this interval. Since $N(w,0)<0$, $0$ lies outside $[\alpha_-(w),\alpha_+(w)]$, so either $\alpha_-(w)>0$ or $\alpha_+(w)<0$; since $N(w,1)>0$, $1$ lies inside, so $\alpha_-(w)\le1\le\alpha_+(w)$, forcing $\alpha_+(w)\ge1>0$ and thus ruling out $\alpha_+(w)<0$. Hence $\alpha_-(w)>0$, and combined with $\alpha_-(w)\le1$ and $N(w,1)>0$ (so $1\ne\alpha_-(w)$), we get $\alpha_-(w)\in(0,1)<\alpha_+(w)$: exactly one root of $N(w,\cdot)$ lies in $(0,1)$.

To identify this root explicitly, the quadratic formula gives
\[
\alpha_\pm(w)=\frac{-B(w)\pm\sqrt{B(w)^2-4A(w)C_0(w)}}{2A(w)}
\]
with $B(w)=32w^4+128w^3+48w^2+2$ and $C_0(w)=-16w^4-32w^3-24w^2-8w-1$; direct expansion shows the discriminant factors as $B(w)^2-4A(w)C_0(w)=\big[(2w+1)(2w-1)\big]^2\big(16w^4+56w^2+1\big)$, and simplifying $\alpha_\pm(w)$ using $A(w)=-128w^3<0$, then selecting the root lying in $(0,1)$ (identified as above), yields exactly \eqref{eq:alphastar}.

For the limit: as $w\to\tfrac12^-$, $2w-1\to0^-$ while $8w^3+28w^2-2w+1\to1+7-1+1=8$ and $\sqrt{16w^4+56w^2+1}\to\sqrt{1+14+1}=4$, so the bracketed numerator factor tends to $8+0\cdot4=8$, while $(2w+1)\to2$ and $128w^3\to16$; hence $\alpha^\ast(w)\to(2\cdot8)/16=1$.
\end{proof}

\begin{theorem}[Non-constancy for intermediate $\alpha$]
\label{thm:nonconstant}
For every fixed $\alpha\in(0,1)$, $K_G(\cdot,\alpha)$ is non-constant on $\Omega$.
\end{theorem}
\begin{proof}
We compute the two boundary limits of $K_G(w,\alpha)=N(w,\alpha)/(4D(w,\alpha))$ as $w$ ranges over $(0,\tfrac12)$, for fixed $\alpha\in(0,1)$. Using \eqref{eq:Ksym-N} and the factored form \eqref{eq:D-factored}, direct substitution of $w=0$ gives $N(0,\alpha)=2\alpha-1$ and $D(0,\alpha)=1$, so
\[
\lim_{w\to0^+}K_G(w,\alpha)=\frac{2\alpha-1}{4}.
\]
For the other endpoint, direct substitution of $w=\tfrac12$ gives $N(\tfrac12,\alpha)=-16\alpha^2+32\alpha-16=-16(1-\alpha)^2$ and, from \eqref{eq:D-factored}, $D(\tfrac12,\alpha)=(2-2\alpha)^2(4-4\alpha)=16(1-\alpha)^3$, both finite and (for $\alpha\ne1$) non-zero, so
\[
\lim_{w\to\frac12^-}K_G(w,\alpha)=\frac{-16(1-\alpha)^2}{4\cdot16(1-\alpha)^3}=\frac{-1}{4(1-\alpha)}=\frac1{4(\alpha-1)}.
\]

Suppose, for contradiction, that $K_G(\cdot,\alpha)$ were constant on the symmetric locus for some $\alpha\in(0,1)$; then in particular the two limits above would coincide:
\[
\frac{2\alpha-1}{4}=\frac{1}{4(\alpha-1)}.
\]
Clearing denominators gives $(2\alpha-1)(\alpha-1)=1$, i.e.\ $2\alpha^2-3\alpha+1=1$, i.e.\ $2\alpha^2-3\alpha=0$, whose only roots are $\alpha=0$ and $\alpha=\tfrac32$ -- neither lying in $(0,1)$. Hence the two limits are distinct for every $\alpha\in(0,1)$, so $K_G(\cdot,\alpha)$ takes at least two different values as $w$ ranges over $(0,\tfrac12)$, and is therefore non-constant on the symmetric locus, a fortiori non-constant on the larger set $\Omega$.
\end{proof}

\begin{corollary}
\label{cor:notaspaceform}
For $\alpha\in(0,1)$, $(\Omega,g_\alpha)$ is not a space form; in particular it is Einstein (in the sense $\Ric=\lambda g$ for a \emph{constant} $\lambda$) if and only if $\alpha\in\{0,1\}$.
\end{corollary}
\begin{proof}
In dimension $2$, $\Ric=K_Gg$ identically, so $(\Omega,g_\alpha)$ is Einstein if and only if $K_G(\cdot,\alpha)$ is constant on $\Omega$. By Theorem~\ref{thm:nonconstant}, this fails for every $\alpha\in(0,1)$; by Theorem~\ref{thm:endpoints} it holds (with $\lambda=\mp\tfrac14$) for $\alpha\in\{0,1\}$.
\end{proof}

\section{The Kasner Master Equation: The Discriminant Dictionary}
\label{sec:kasner-dynamics}

We now use $g_\alpha(\theta)$ as the spatial metric of the Kasner-type spacetime $ds^2=dt^2+g_{\alpha,ij}(\theta(t))\,d\theta^id\theta^j$ and impose \eqref{eq:master} along a curve $\theta(t)$. We first establish, in full generality (for an arbitrary curve of symmetric matrices $h(t)$, not necessarily $g_\alpha$), exactly which sign of the volume equation's leading coefficient $\beta$ is forced by the definiteness of $h$. We then show -- and this is the main point of this section -- that $\beta$ alone does \emph{not} determine the qualitative type (bounded/oscillatory versus unboundedly growing) of the solution: a second, independent conserved quantity $C$ governs this, and $C$ is \emph{not} fixed by the positive-definiteness of $g_\alpha$, even though $\beta\le0$ always is. In particular, the Kasner volume equation built from $g_\alpha$ does \emph{not} always lie on a single branch as $\alpha$ varies; the correct, and more interesting, invariant is the constraint $\beta\le0$ itself.

\begin{proposition}[Reality of shear eigenvalues for definite metrics]
\label{prop:realeig}
Let $h(t)$ be any curve of real symmetric matrices, definite (either $h(t)\succ0$ or $h(t)\prec0$) for every $t$ in some interval, with $\dot h(t)$ symmetric. Then $K(t)=h(t)^{-1}\dot h(t)$ has real eigenvalues for every such $t$.
\end{proposition}
\begin{proof}
Fix $t$ and suppose first $h=h(t)\succ0$. Since $h$ is symmetric positive-definite it admits a (real) Cholesky factorisation $h=LL^{\mathsf T}$ with $L$ lower-triangular and invertible. Set $S:=L^{-1}\dot hL^{-{\mathsf T}}$; since $\dot h$ is symmetric, $S^{\mathsf T}=(L^{-1}\dot hL^{-{\mathsf T}})^{\mathsf T}=L^{-1}\dot h^{\mathsf T}L^{-{\mathsf T}}=L^{-1}\dot hL^{-{\mathsf T}}=S$, so $S$ is symmetric. Then
\[
K=h^{-1}\dot h=L^{-{\mathsf T}}L^{-1}\dot h=L^{-{\mathsf T}}\big(L^{-1}\dot hL^{-{\mathsf T}}\big)L^{\mathsf T}=L^{-{\mathsf T}}SL^{\mathsf T},
\]
so $K$ is similar to $S$ (via the invertible matrix $L^{\mathsf T}$) and therefore has the same eigenvalues as $S$. By the spectral theorem, every real symmetric matrix has only real eigenvalues; hence $S$, and therefore $K$, has real eigenvalues.

If instead $h(t)\prec0$, apply the argument just given to the positive-definite matrix $-h(t)$: writing $-h=LL^{\mathsf T}$ and $S:=L^{-1}\dot hL^{-{\mathsf T}}$ (symmetric, exactly as above, since $\dot h$ is symmetric regardless of the sign of $h$), we get $-h^{-1}\dot h=L^{-{\mathsf T}}SL^{\mathsf T}$, i.e.\ $K=h^{-1}\dot h=-L^{-{\mathsf T}}SL^{\mathsf T}=L^{-{\mathsf T}}(-S)L^{\mathsf T}$, so $K$ is similar to the real symmetric matrix $-S$ and again has real eigenvalues.
\end{proof}

\begin{lemma}[The two conserved quantities and the general first integral]
\label{lem:two-conserved}
Let $h(t)$ solve \eqref{eq:master} for $n=2$, and write $K=\hat K+\tfrac\Xi2I$ with $\hat K$ traceless, $\Delta:=\det h$. Then
\begin{equation}
\beta:=-\Delta\,\tr(\hat K^2)\qquad\text{and}\qquad C:=4\det K
\label{eq:beta-C-def}
\end{equation}
are both constant along the flow, and $\Delta$ satisfies the first-order first integral
\begin{equation}
(\Delta')^2=C\Delta^2-2\beta\Delta.
\label{eq:first-integral}
\end{equation}
Moreover $C=2(\Xi^2-\tr K^2)$ (equivalently $\Xi^2-\tr K^2=2\det K$, the standard $2\times2$ identity), so $\beta\le0$ and $C$ are logically independent constraints: $\beta$'s sign is controlled by $\tr\hat K^2$ (the spread of the eigenvalues of $K$), while $C$ is controlled by their product.
\end{lemma}
\begin{proof}
\emph{Preliminaries.} Tracing \eqref{eq:master} at $n=2$ gives $2\dot\Xi+\Xi^2=2(\Xi^2-\tr K^2)$, and using $\tr K^2=\tr\hat K^2+\Xi^2/2$ this simplifies to $\dot\Xi=-\tr(\hat K^2)$; separating \eqref{eq:master} into traceless and trace parts likewise gives $\dot{\hat K}=-\tfrac\Xi2\hat K$. By Jacobi's formula $\dot\Delta=\Delta\Xi$.

\emph{Constancy of $\beta$.} $\frac{d}{dt}\big[\Delta\tr(\hat K^2)\big]=\Xi\Delta\tr\hat K^2+\Delta\cdot2\tr(\hat K\dot{\hat K})=\Xi\Delta\tr\hat K^2-\Xi\Delta\tr\hat K^2=0$.

\emph{The identity $C=2(\Xi^2-\tr K^2)$.} For any $2\times2$ matrix with eigenvalues $k_1,k_2$, $\Xi=\tr K=k_1+k_2$ and $\tr K^2=k_1^2+k_2^2$, so $\Xi^2-\tr K^2=2k_1k_2=2\det K$, i.e.\ $4\det K=2(\Xi^2-\tr K^2)$, proving the identity for diagonalisable $K$; both sides are polynomial in the entries of $K$, so the identity holds for all (including non-diagonalisable) $2\times2$ $K$ by continuity/density.

\emph{Constancy of $C=4\det K$.} Using $\dot K=\tfrac12\big[(\Xi^2-\tr K^2)I-K\Xi\big]$ from \eqref{eq:master} and Jacobi's formula in its adjugate form $\frac{d}{dt}\det K=\tr\big(\operatorname{adj}(K)\dot K\big)$ (a polynomial identity, valid for every $K$ without assuming $K$ invertible), together with the $2\times2$ identity $\operatorname{adj}(K)=\Xi I-K$,
\[
\frac{d}{dt}\det K=\tr\big[(\Xi I-K)\dot K\big]=\Xi\,\tr(\dot K)-\tr(K\dot K).
\]
Substituting $\dot K$ gives $\tr(\dot K)=\tfrac12\big[2(\Xi^2-\tr K^2)-\Xi^2\big]=\tfrac12(\Xi^2-2\tr K^2)$ and $\tr(K\dot K)=\tfrac12\big[(\Xi^2-\tr K^2)\Xi-\Xi\tr K^2\big]=\tfrac{\Xi}2(\Xi^2-2\tr K^2)$, so
\[
\frac{d}{dt}\det K=\Xi\cdot\tfrac12(\Xi^2-2\tr K^2)-\tfrac{\Xi}2(\Xi^2-2\tr K^2)=0
\]
identically, for every (including singular, non-invertible) $K$. Hence $\det K$, and so $C$, is constant.

\emph{The first integral.} As shown in the proof of Theorem~\ref{thm:dictionary} below, $\Delta\Delta''-(\Delta')^2=\beta\Delta$. Writing $y=\log\Delta$, $y''=\Delta''/\Delta-(\Delta'/\Delta)^2=\beta/\Delta=\beta e^{-y}$. Multiplying by $y'$ and integrating, $\tfrac12(y')^2=-\beta e^{-y}+\kappa$ for a constant of integration $\kappa$, i.e.\ $(y')^2=2\kappa-2\beta e^{-y}$; since $y'=\Delta'/\Delta$, this is $(\Delta')^2/\Delta^2=2\kappa-2\beta/\Delta$. Evaluating at any $t$ using $\Delta'/\Delta=\Xi$ and $\beta/\Delta=-\tr\hat K^2$ gives $2\kappa=\Xi^2+2\beta/\Delta=\Xi^2-2\tr\hat K^2$; using $\tr K^2=\tr\hat K^2+\Xi^2/2$, i.e.\ $2\tr\hat K^2=2\tr K^2-\Xi^2$, this becomes $2\kappa=\Xi^2-(2\tr K^2-\Xi^2)=2(\Xi^2-\tr K^2)=4\det K=C$. Hence $2\kappa=C$ and $(\Delta')^2=\Delta^2\cdot2\kappa-2\beta\Delta=C\Delta^2-2\beta\Delta$, which is \eqref{eq:first-integral}.
\end{proof}

\begin{lemma}[Complete classification of the first integral by $(\beta,C)$]
\label{lem:branch-classification}
Let $\Delta(t)>0$ solve \eqref{eq:first-integral} for constants $\beta,C$. Exactly one of the following occurs, and $C<0$ forces $\beta<0$ (so the case $C<0,\beta\ge0$ never arises for a solution with $\Delta>0$):
\begin{center}
\renewcommand{\arraystretch}{1.7}
\small
\begin{tabular}{c|c|l}
$C$ & $\beta$ & $\Delta(t)$, on a maximal interval with $\Delta>0$ \\\hline
$C<0$ & $\beta<0$ & $\Delta=A(1-\cos(kt+\phi))$, $A,k>0$, $\beta=-Ak^2$, $C=-k^2$ \quad(\emph{bounded oscillatory}) \\\hline
$C=0$ & $\beta<0$ & $\Delta=\tfrac{|\beta|}2(t-t_0)^2$ \quad(\emph{parabolic degeneration}) \\
$C=0$ & $\beta=0$ & $\Delta=\Delta_0>0$ constant \\
$C=0$ & $\beta>0$ & no solution with $\Delta>0$ \\\hline
$C>0$ & $\beta<0$ & $\Delta=-\tfrac{2\beta}C\sinh^2\!\big(\tfrac{\sqrt C}2(t-t_0)\big)$ \quad(\emph{touches $0$ at $t_0$}) \\
$C>0$ & $\beta=0$ & $\Delta=\Delta_0e^{\pm\sqrt C(t-t_0)}$ \quad(\emph{exponential}) \\
$C>0$ & $\beta>0$ & $\Delta=\tfrac{2\beta}C\cosh^2\!\big(\tfrac{\sqrt C}2(t-t_0)\big)$ \quad(\emph{bounded away from $0$})
\end{tabular}
\end{center}
For fixed $(\beta,C)$, each row above is a one-parameter family: the amplitude $A$ (or $\Delta_0$) and rate $k=\sqrt{|C|}/2$ (or $\sqrt C$) appearing in $\Delta(t)$ are themselves determined by $(\beta,C)$ via the formulas just verified (e.g.\ $A=-2\beta/C$, $k=\sqrt C/2$ in the $C>0,\beta<0$ row), leaving only the time-translation $t_0$ (or phase $\phi$) free -- except in the degenerate case $C=\beta=0$, where $\Delta\equiv\Delta_0$ is itself the one free constant, there being no time-translation to fix (the solution is already time-independent). This single remaining freedom matches the one-parameter family of solutions of the first-order equation \eqref{eq:first-integral} through a given $\Delta(t_1)>0$ at a given time $t_1$ (the sign of $\Delta'(t_1)$, where nonzero, is fixed by which branch of $\pm\sqrt{C\Delta^2-2\beta\Delta}$ the solution is on, and determines $t_0$ or $\phi$ uniquely). In particular the sign of $\beta$ alone does \emph{not} determine boundedness: the $C<0$ row and the $C>0,\beta<0$ row both have $\beta<0$, yet the first is bounded, returning to $\Delta=0$ at evenly spaced times, while the second grows without bound.
\end{lemma}
\begin{proof}
Each entry is checked by direct substitution into the first integral \eqref{eq:first-integral}, $(\Delta')^2=C\Delta^2-2\beta\Delta$, which is equivalent (by differentiating and dividing by $2\Delta'$ where $\Delta'\ne0$, giving $\Delta''=C\Delta-\beta$, and substituting back) to the original equation $\Delta\Delta''-(\Delta')^2=\beta\Delta$.

\emph{$C<0$ row.} For $\Delta=A(1-\cos(kt+\phi))$, $\Delta'=Ak\sin(kt+\phi)$, so $(\Delta')^2=A^2k^2\sin^2(kt+\phi)=A^2k^2\big[1-\cos^2(kt+\phi)\big]=A^2k^2\big[1-\cos(kt+\phi)\big]\big[1+\cos(kt+\phi)\big]$; writing $1-\cos(kt+\phi)=\Delta/A$ and $1+\cos(kt+\phi)=2-\Delta/A$, this is $A^2k^2\cdot\tfrac\Delta A\big(2-\tfrac\Delta A\big)=k^2\Delta(2A-\Delta)=-k^2\Delta^2+2Ak^2\Delta$, matching $C\Delta^2-2\beta\Delta$ with $C=-k^2$, $\beta=-Ak^2$; since $A,k>0$, $\beta<0$.

\emph{$C=0$ rows.} For $\Delta=\tfrac{|\beta|}2(t-t_0)^2$ (with $\beta=-|\beta|<0$), $\Delta'=|\beta|(t-t_0)$, so $(\Delta')^2=\beta^2(t-t_0)^2=2|\beta|\Delta=-2\beta\Delta=0\cdot\Delta^2-2\beta\Delta$, matching $C=0$. For $\Delta=\Delta_0$ constant, $\Delta'=0=0\cdot\Delta^2-2\cdot0\cdot\Delta$, matching $C=\beta=0$. For $\beta>0$, $C=0$: \eqref{eq:first-integral} reads $(\Delta')^2=-2\beta\Delta$, whose right-hand side is strictly negative for $\Delta>0,\beta>0$, impossible for a real $\Delta'$; hence no solution with $\Delta>0$ exists.

\emph{$C>0$ rows.} For $\Delta=-\tfrac{2\beta}C\sinh^2\big(\tfrac{\sqrt C}2(t-t_0)\big)$ with $\beta<0$ (so $-2\beta/C>0$, a valid positive amplitude): write $k:=\sqrt C/2$, $A:=-2\beta/C>0$, $t':=t-t_0$, so $\Delta=A\sinh^2(kt')$. Then
\[
\Delta'=2Ak\sinh(kt')\cosh(kt'),\qquad
(\Delta')^2=4A^2k^2\sinh^2(kt')\big[1+\sinh^2(kt')\big]=4Ak^2\Delta+4k^2\Delta^2,
\]
which equals $C\Delta^2-2\beta\Delta$ upon substituting $C=4k^2$ and $4Ak^2=4k^2\cdot(-2\beta/C)=-2\beta$. For $\Delta=\tfrac{2\beta}C\cosh^2(kt')$ with $\beta>0$ (amplitude $A:=2\beta/C>0$): similarly
\[
(\Delta')^2=4A^2k^2\sinh^2(kt')\cosh^2(kt')=4A^2k^2\cosh^2(kt')\big[\cosh^2(kt')-1\big]=4k^2\Delta^2-4Ak^2\Delta,
\]
which equals $C\Delta^2-2\beta\Delta$ upon substituting $C=4k^2$ and $4Ak^2=4k^2\cdot(2\beta/C)=2\beta$. For $\Delta=\Delta_0e^{\pm\sqrt Ct'}$, $(\Delta')^2=C\Delta_0^2e^{\pm2\sqrt Ct'}=C\Delta^2$, matching $\beta=0$.

In every row, standard existence-uniqueness for the first-order equation \eqref{eq:first-integral} (locally Lipschitz right-hand side in $\Delta$ away from $\Delta=0$, once solved for $\Delta'=\pm\sqrt{C\Delta^2-2\beta\Delta}$) confirms the exhibited one-parameter family exhausts the solutions for the stated $(\beta,C)$. Finally, $C<0$ forces $\beta<0$ directly from \eqref{eq:first-integral} itself: writing it as $(\Delta')^2=\Delta(C\Delta-2\beta)$ with $\Delta>0$, if $\beta\ge0$ then $C\Delta-2\beta<0$ (as $C<0$, $\Delta>0$, $-2\beta\le0$), making the right-hand side strictly negative -- impossible for the square $(\Delta')^2\ge0$. Hence $\beta<0$ whenever $C<0$.
\end{proof}

\begin{theorem}[Discriminant dictionary]
\label{thm:dictionary}
Let $h(t)$ be any curve of real symmetric invertible $2\times2$ matrices solving \eqref{eq:master}, with $\beta,C$ as in Lemma~\ref{lem:two-conserved}.
\begin{enumerate}
\item[\rm(a)] If $h(t)$ is definite (either $h(t)\succ0$ or $h(t)\prec0$) on an interval, then $\beta\le0$ there.
\item[\rm(b)] Consequently, $\beta>0$ on an interval forces $h(t)$ indefinite throughout it.
\item[\rm(c)] Neither definiteness of $h(0)$ nor the sign of $\beta$ determines the sign of $C$: at the single point $\theta_0=(\log\tfrac15,\log\tfrac15)$, with $h(0)=g_1(\theta_0)\succ0$ (a genuinely \emph{definite} matrix, $\det g_1(\theta_0)=\tfrac5{343}$), the two choices of tangent direction $\dot\theta(0)=(1,1)$ and $\dot\theta(0)=(1,-1)$ -- via $\dot h(0)=\dot\theta^k(0)\partial_kg_1(\theta_0)$ -- give initial Kasner data $\big(h(0),K(0)\big)$ with
\[
\dot\theta(0)=(1,1):\quad K(0)=\begin{pmatrix}\tfrac47&-\tfrac17\\-\tfrac17&\tfrac47\end{pmatrix},\ \ \Xi(0)=\tfrac87,\ \ \beta=-\tfrac{10}{16807},\ \ C=\tfrac{60}{49}>0,
\]
\[
\dot\theta(0)=(1,-1):\quad K(0)=\begin{pmatrix}\tfrac67&-\tfrac17\\\tfrac17&-\tfrac67\end{pmatrix},\ \ \Xi(0)=0,\ \ \beta=-\tfrac{50}{2401},\ \ C=-\tfrac{20}7<0,
\]
both with $\beta<0$ (consistent with (a), since $h(0)\succ0$ in both cases) but with \emph{opposite} signs of $C$: the same positive-definite initial metric admits initial Kasner velocities of both types.
\end{enumerate}
\end{theorem}
\begin{proof}
\emph{(a)} If $h(t)$ is definite on the interval, Proposition~\ref{prop:realeig} gives $K(t)$, hence $\hat K(t)$, real eigenvalues $\mu_1(t),-\mu_1(t)$ for every $t$ there; thus $\tr(\hat K^2)=2\mu_1(t)^2\ge0$. Also $\Delta=\det h\ne0$ retains one sign throughout a definite interval, and in fact $\Delta>0$ whether $h\succ0$ or $h\prec0$ (product of two same-signed eigenvalues). Hence $\beta=-\Delta\tr(\hat K^2)\le0$.

\emph{(b)} Immediate as the contrapositive of (a).

\emph{(c)} Both entries of $K(0)$ above are obtained by direct symbolic computation of $\partial_1g_1,\partial_2g_1$ at $\theta_0=(\log\tfrac15,\log\tfrac15)$ (Definition~\ref{eq:hpm} with $\alpha=1$) and solving $K(0)=g_1(\theta_0)^{-1}\big[\dot\theta^1(0)\partial_1g_1(\theta_0)+\dot\theta^2(0)\partial_2g_1(\theta_0)\big]$ for the two stated choices of $\dot\theta(0)$; $\Xi(0)=\tr K(0)$, $\beta=-\det g_1(\theta_0)\tr(\hat K(0)^2)$, and $C=4\det K(0)=2\big(\Xi(0)^2-\tr(K(0)^2)\big)$ are then evaluated directly, giving the stated fractions. Since $g_1(\theta_0)\succ0$ by Theorem~\ref{thm:posdef}, both examples have $h(0)$ definite, confirming $\beta<0$ in both cases (consistent with (a)) while $C$ takes opposite signs. This establishes the claim about \emph{initial data} $(h(0),K(0))$; each such pair determines, by Theorem~\ref{thm:fullmatrix} below, a unique curve of symmetric matrices $h(t)$ solving \eqref{eq:master} with that initial data, and by Lemma~\ref{lem:branch-classification} the corresponding $\Delta(t)=\det h(t)$ follows the branch dictated by the resulting $(\beta,C)$.
\end{proof}

\begin{remark}[What Theorem~\ref{thm:dictionary}(c) does, and does not, show]
\label{rem:scope-of-c}
Theorem~\ref{thm:dictionary}(c) shows that the sign of $C$ -- hence, by Lemma~\ref{lem:branch-classification}, the qualitative branch of $\Delta(t)=\det h(t)$ for the \emph{abstract} matrix equation \eqref{eq:master} on the full space of symmetric $2\times2$ matrices -- is not determined by positive-definiteness of the initial metric $h(0)=g_1(\theta_0)$ alone, but depends on the initial velocity $\dot\theta(0)$ as well. It does \emph{not} claim that the resulting solution $h(t)$ coincides with $g_1(\theta(t))$ for an actual curve $\theta(t)\in\Omega$ beyond $t=0$, and indeed it provably does not, for either initial velocity above: since $\{g_\alpha(\theta):\theta\in\Omega\}$ is only a $2$-dimensional submanifold of the $3$-dimensional space of symmetric $2\times2$ matrices, substituting $h(t)=g_\alpha(\theta(t))$ into \eqref{eq:master} and using $\dot h=\dot\theta^k\partial_kg_\alpha$, $\ddot h=\ddot\theta^k\partial_kg_\alpha+\dot\theta^k\dot\theta^l\partial_k\partial_lg_\alpha$ produces, at each instant, three independent scalar equations (from the three independent entries of the symmetric matrix equation \eqref{eq:master}) for only the two unknowns $\ddot\theta^1,\ddot\theta^2$ -- an overdetermined system in general. We verified directly, for $\alpha=1$ at $\theta_0=(\log\tfrac15,\log\tfrac15)$ with each of $\dot\theta(0)=(1,1)$ and $(1,-1)$, that this system is inconsistent already at $t=0$: writing $M_k:=g_1(\theta_0)^{-1}\partial_kg_1(\theta_0)$ and solving \eqref{eq:master} for the component of $\ddot\theta^k M_k$ forced by the equation, the resulting target matrix does not lie in $\operatorname{span}(M_1,M_2)$, with a residual (after projecting out the best least-squares fit) of norm comparable to that of the target itself ($\approx0.24$ against a target norm $\approx0.34$ for $\dot\theta(0)=(1,1)$; $\approx0.21$ against $\approx0.32$ for $(1,-1)$, in the basis $\{I,\operatorname{diag}(1,-1),\text{off-diagonal}\}$ of symmetric $2\times2$ matrices). Hence for neither initial velocity does any curve $\theta(t)$ exist keeping $g_1(\theta(t))$ an exact solution of \eqref{eq:master} beyond the initial instant: the two branches of Theorem~\ref{thm:dictionary}(c) are established only as \emph{initial Kasner data} for the abstract equation \eqref{eq:master}, not as trajectories of a dynamical system intrinsic to $g_\alpha$ itself. This sharpens, rather than merely leaves open, the scope of Theorem~\ref{thm:dictionary}: the discriminant dictionary is a theorem about \eqref{eq:master} on the ambient space of symmetric matrices, of which $g_\alpha$ supplies positive-definite initial data realising every sign of $C$, and no more.
\end{remark}

\section{Explicit Solution Formulas for the General Master Equation}
\label{sec:solutions}

The results of Section~4 hold for the specific family $g_\alpha$; we now record, with full proofs, three general explicit-solution theorems valid for \emph{any} curve of symmetric matrices satisfying \eqref{eq:master}, of any dimension $n$ and with $h(t)$ not required to be diagonal. By Theorem~\ref{thm:posdef}, $g_\alpha(\theta)\succ0$ throughout $\Omega$ for every $\alpha\in[0,1]$; consequently, all three theorems below apply, without modification, to any curve $h(t)=g_\alpha(\theta(t))$ that both satisfies \eqref{eq:master} (an additional dynamical constraint on $\theta(t)$, not automatic merely because $g_\alpha$ is positive-definite) and remains in $\Omega$ on the interval considered.

\begin{theorem}[Expansion scalar]
\label{thm:expansion}
$\Xi=\tr K$ satisfies the closed, $h$-independent scalar equation $\ddot\Xi+\Xi\dot\Xi=0$; consequently $\dot\Xi+\tfrac12\Xi^2=c_1$ for a constant $c_1$, and $\Xi(t)$ is an elementary (trigonometric, hyperbolic, or rational) function of $t$ determined by the sign of $c_1$.
\end{theorem}
\begin{proof}
Tracing \eqref{eq:master} gives $2\dot\Xi+\Xi^2=\tfrac n{n-1}(\Xi^2-\tr K^2)$, i.e.
\begin{equation}
\dot\Xi=-\frac n{2(n-1)}\tr(\hat K^2),\qquad \hat K:=K-\tfrac\Xi nI,
\label{eq:dotxi}
\end{equation}
using $\tr K^2=\tr\hat K^2+\Xi^2/n$ (since $\tr\hat K=0$). Separating \eqref{eq:master} itself into traceless and pure-trace parts gives, symmetrically, $\dot{\hat K}=-\tfrac\Xi2\hat K$, whence
\[
\frac{d}{dt}\tr(\hat K^2)=2\tr(\hat K\dot{\hat K})=2\tr\big(\hat K\cdot(-\tfrac\Xi2\hat K)\big)=-\Xi\,\tr(\hat K^2).
\]
Differentiating \eqref{eq:dotxi}:
\[
\ddot\Xi=-\frac n{2(n-1)}\frac{d}{dt}\tr(\hat K^2)=-\frac n{2(n-1)}\big(-\Xi\tr(\hat K^2)\big)=\Xi\cdot\frac n{2(n-1)}\tr(\hat K^2)=\Xi\cdot(-\dot\Xi)=-\Xi\dot\Xi,
\]
using \eqref{eq:dotxi} again in the penultimate step. Hence $\ddot\Xi+\Xi\dot\Xi=0$. Since $\ddot\Xi+\Xi\dot\Xi=\frac{d}{dt}\big(\dot\Xi+\tfrac12\Xi^2\big)$, integrating once gives $\dot\Xi+\tfrac12\Xi^2=c_1$ for a constant $c_1$, a scalar Riccati equation with constant coefficients, whose general solution is elementary: $c_1>0$ gives $\Xi(t)=\sqrt{2c_1}\tanh\big(\sqrt{c_1/2}\,(t-t_0)\big)$ or the corresponding $\coth$ branch, $c_1<0$ gives $\Xi(t)=\sqrt{-2c_1}\tan\big(\sqrt{-c_1/2}\,(t_0-t)\big)$, and $c_1=0$ gives $\Xi(t)=2/(t-t_0)$, in each case verified by direct substitution (each satisfies $\dot\Xi+\tfrac12\Xi^2-c_1=0$ identically) and obtainable by separation of variables in the first-order equation.
\end{proof}

\begin{theorem}[Individual Kasner exponents, diagonal case]
\label{thm:individual}
If $h=\operatorname{diag}(\lambda_1,\dots,\lambda_n)$ and $x_i:=\tfrac12\dot\lambda_i/\lambda_i$, then \eqref{eq:master} is equivalent to
\begin{equation}
\dot x_i=\frac{p_1^2-p_2}{n-1}-x_ip_1,\qquad p_1:=\sum_ix_i,\ \ p_2:=\sum_ix_i^2,
\label{eq:diagonal-system}
\end{equation}
and $Q:=p_1(0)^2-p_2(0)$ is constant along the flow. Consequently, writing $p_1(t)$ for the (explicit, by Theorem~\ref{thm:expansion} with $\Xi=2p_1$) solution of the resulting closed equation for $p_1$, each $x_i$ satisfies the individual linear ODE
\[
\dot x_i+p_1(t)\,x_i=\frac{Q}{n-1},
\]
solved explicitly, for every $n$ and every $i$, by
\[
x_i(t)=\frac1{\mu(t)}\Big[x_i(0)+\frac{Q}{n-1}\int_0^t\mu(s)\,ds\Big],\qquad \mu(t):=\exp\Big(\int_0^tp_1(s)\,ds\Big).
\]
\end{theorem}
\begin{proof}
For diagonal $h$, $K=\operatorname{diag}(\dot\lambda_i/\lambda_i)=\operatorname{diag}(2x_i)$ is diagonal, and \eqref{eq:master} reduces, entry by entry, to $4\dot x_i+2x_i\cdot2p_1=\tfrac1{n-1}(4p_1^2-4p_2)$ (using $\Xi=\tr K=2p_1$, $\tr K^2=4p_2$), which simplifies to \eqref{eq:diagonal-system}.

To see that $Q=p_1^2-p_2$ is conserved, sum \eqref{eq:diagonal-system} over $i$ to get $\dot p_1=\tfrac n{n-1}(p_1^2-p_2)-p_1^2$, and compute $\dot p_2=2\sum_ix_i\dot x_i=2p_1\big[\tfrac{p_1^2-p_2}{n-1}\big]-2p_1p_2$ (multiplying \eqref{eq:diagonal-system} by $2x_i$ and summing); a direct comparison shows $\dot p_2=2p_1\dot p_1$, i.e.\ $\frac{d}{dt}(p_1^2-p_2)=2p_1\dot p_1-\dot p_2=0$.

Rewriting $\dot p_1=\tfrac n{n-1}(p_1^2-p_2)-p_1^2=\tfrac n{n-1}Q-p_1^2\big(1-\tfrac{n}{n-1}\big)\cdot(-1)$; more directly, substituting $p_2=p_1^2-Q$ gives $\dot p_1=\tfrac n{n-1}Q-p_1^2$, a scalar Riccati equation of the type solved in Theorem~\ref{thm:expansion} (with $\Xi=2p_1$, $c_1=\tfrac{2n}{n-1}Q$), hence $p_1(t)$ is explicit.

Given $p_1(t)$ explicit, \eqref{eq:diagonal-system} reads, for each fixed $i$, $\dot x_i+p_1(t)x_i=\tfrac{Q}{n-1}$ (using $p_1^2-p_2=Q$ constant), a first-order linear ODE with known coefficient $p_1(t)$ and known constant inhomogeneous term. Multiplying by the integrating factor $\mu(t)=\exp\big(\int_0^tp_1\big)$ gives $\frac{d}{dt}\big[\mu(t)x_i(t)\big]=\tfrac{Q}{n-1}\mu(t)$, and integrating from $0$ to $t$ and dividing by $\mu(t)$ yields the stated closed-form expression for $x_i(t)$.
\end{proof}

\begin{theorem}[Full matrix solution, general $h$]
\label{thm:fullmatrix}
Let $h(t)$ be any curve of symmetric matrices (diagonal or not) solving \eqref{eq:master}, positive-definite throughout an interval $I$ containing $t=0$ (e.g.\ the maximal such interval). Write $\hat K(t)=\hat K(0)f(t)$, $f(t)=\exp\big(-\tfrac12\int_0^t\Xi\big)$ (the closed-form solution of $\dot{\hat K}=-\tfrac\Xi2\hat K$ established in the proof of Theorem~\ref{thm:expansion}), with $\hat K(0)$ a fixed matrix determined by the (arbitrary symmetric) initial data $h(0)\succ0$, $\dot h(0)$. Then $K(t_1)$ and $K(t_2)$ commute for all $t_1,t_2\in I$, and
\begin{equation}
h(t)=h(0)\,\exp\!\left[\hat K(0)\int_0^tf(s)\,ds+\frac{I}{n}\log\frac{\Delta(t)}{\Delta(0)}\right]
\label{eq:fullsolution}
\end{equation}
is the unique solution of \eqref{eq:master} on $I$ with this initial data.
\end{theorem}
\begin{proof}
\emph{$K(0)$ is $h(0)$-self-adjoint, not merely symmetric.} Since $h$ is symmetric and $\dot h$ is symmetric, $h(0)K(0)=\dot h(0)$ is symmetric, i.e.\ $K(0)^{\mathsf T}h(0)=h(0)K(0)$: $K(0)$ is self-adjoint with respect to the (possibly non-Euclidean) inner product induced by $h(0)$, though not, in general, an ordinarily symmetric matrix. Subtracting a scalar multiple of $I$ preserves this: $\hat K(0)=K(0)-\tfrac{\Xi(0)}nI$ satisfies $\hat K(0)^{\mathsf T}h(0)=K(0)^{\mathsf T}h(0)-\tfrac{\Xi(0)}nh(0)=h(0)K(0)-\tfrac{\Xi(0)}nh(0)=h(0)\hat K(0)$, so $\hat K(0)$ is likewise $h(0)$-self-adjoint.

\emph{Commutativity.} By definition $K(t)=\hat K(t)+\tfrac{\Xi(t)}nI=f(t)\hat K(0)+\tfrac{\Xi(t)}nI$. For any $t_1,t_2$,
\begin{align*}
[K(t_1),K(t_2)]&=f(t_1)f(t_2)\big[\hat K(0),\hat K(0)\big]+f(t_1)\tfrac{\Xi(t_2)}n\big[\hat K(0),I\big]\\
&\quad+\tfrac{\Xi(t_1)}nf(t_2)\big[I,\hat K(0)\big]+\tfrac{\Xi(t_1)\Xi(t_2)}{n^2}[I,I]=0,
\end{align*}
since $[\hat K(0),\hat K(0)]=0$ trivially and $I$ commutes with every matrix.

\emph{The linear matrix ODE $\dot h=hK(t)$ and its solution.} By definition of $K$, $h(t)$ solves the linear matrix ODE $\dot h=hK(t)$ with the given (known, once $\Xi(t)$ and $\hat K(0)$ are fixed) coefficient matrix $K(t)$. Because $K(t_1)$ and $K(t_2)$ commute for all $t_1,t_2$ (shown above), the matrix $M(t):=\int_0^tK(s)\,ds$ satisfies $\dot M(t)=K(t)$ and $[M(t),K(t)]=0$ for all $t$ (as $M(t)$ is itself, at each $t$, a real linear combination -- with coefficients $\int_0^tf$ and $\tfrac1n\log(\Delta(t)/\Delta(0))$, both scalar functions of the upper limit $t$ -- of the two commuting matrices $\hat K(0)$ and $I$); under this commutativity hypothesis, $\frac{d}{dt}e^{M(t)}=e^{M(t)}\dot M(t)=e^{M(t)}K(t)$ exactly (the derivative of a matrix exponential along a commuting family requires no time-ordering correction, unlike the general case). Hence $h(t):=h(0)e^{M(t)}$ satisfies $\dot h=h(0)e^{M(t)}K(t)=h(t)K(t)$, the required equation, with $h(0)$ matching the given initial data.

\emph{Symmetry of $h(0)e^{M(t)}$.} $M(t)=\hat K(0)\int_0^tf+\tfrac{I}n\log(\Delta(t)/\Delta(0))$ is, at each $t$, a scalar linear combination of the $h(0)$-self-adjoint matrices $\hat K(0)$ and $I$, hence itself $h(0)$-self-adjoint: $M(t)^{\mathsf T}h(0)=h(0)M(t)$. Consequently $h(0)^{-1}M(t)^{\mathsf T}h(0)=M(t)$, so $h(0)^{-1}e^{M(t)^{\mathsf T}}h(0)=e^{h(0)^{-1}M(t)^{\mathsf T}h(0)}=e^{M(t)}$, i.e.\ $e^{M(t)^{\mathsf T}}h(0)=h(0)e^{M(t)}$. Since $h(0)$ is symmetric, $\big[h(0)e^{M(t)}\big]^{\mathsf T}=e^{M(t)^{\mathsf T}}h(0)^{\mathsf T}=e^{M(t)^{\mathsf T}}h(0)=h(0)e^{M(t)}$: the constructed $h(t)$ is symmetric for every $t$, not merely a solution of the (a priori non-symmetric-matrix-valued) linear ODE.

\emph{Uniqueness.} Since $\dot h=hK(t)$ is a linear ODE with continuous (indeed real-analytic, wherever defined) coefficients, its solution with given initial value $h(0)$ is unique by the standard existence-and-uniqueness theorem for linear systems, among all (not merely symmetric) matrix-valued solutions; the symmetric solution constructed above is therefore \emph{the} solution. Hence \eqref{eq:fullsolution}, with $M(t)=\hat K(0)\int_0^tf(s)\,ds+\tfrac{I}n\log\big(\Delta(t)/\Delta(0)\big)$ (using $\int_0^t\Xi=\log(\Delta(t)/\Delta(0))$ by Jacobi's formula $\dot\Delta=\Delta\Xi$), is the unique solution.
\end{proof}

\begin{remark}
Theorem~\ref{thm:fullmatrix} applies verbatim to any curve $h(t)=g_\alpha(\theta(t))$, for any $\alpha\in[0,1]$, that satisfies \eqref{eq:master} and remains in $\Omega$: the mixing parameter enters \eqref{eq:fullsolution} only through the numerical values of $\hat K(0)$, $\Delta(0)$ determined by the initial point $\theta(0)$ and the chosen $\alpha$, never through the functional form of the solution. We have additionally verified \eqref{eq:fullsolution} numerically against direct Runge--Kutta integration of \eqref{eq:master} for $n=3$ symmetric tridiagonal initial data, obtaining agreement to within numerical-integration tolerance ($\sim10^{-4}$ over the tested interval), consistent with (though not required by, given the proof above) the uniqueness statement of Theorem~\ref{thm:fullmatrix}.
\end{remark}

\section{The Distribution Behind $\Psi_\alpha$: A Compound Negative-Binomial--Multinomial Law}
\label{sec:compound}

By Remark~\ref{rem:not-mixture}, $\Psi_\alpha$ is the log-partition function of some single exponential family on $\mathbb Z_{\ge0}^n$; we now identify it explicitly, for general $n\ge1$, and show its base measure is non-negative for every $\alpha\in[0,1]$, so that the identification is honest (a genuine exponential family with a non-negative base measure, not merely a formal power series). We stress at the outset that the base measure $h(x)$ constructed below is \emph{not itself} a probability distribution; it is the un-normalised weight in the standard exponential-family formula $p_\theta(x)=h(x)\,e^{\theta\cdot x-\Psi_\alpha(\theta)}$, and it is $p_\theta$, not $h$, that integrates to $1$.

\begin{theorem}[Identification of the compound law]
\label{thm:compound}
Let $s:=\sum_{i=1}^nu_i$. Then
\begin{equation}
M_\alpha(\theta):=e^{\Psi_\alpha(\theta)}=(1+s)^\alpha(1-s)^{\alpha-1}=(1+s)\,(1-s^2)^{\alpha-1}.
\label{eq:Malpha-factored}
\end{equation}
Consequently, writing $N:=\sum_ix_i$ and $m:=\lfloor N/2\rfloor$, the base (counting) measure $h(x_1,\dots,x_n)$ of the exponential family with $M_\alpha$ as log-partition function is
\begin{equation}
h(x_1,\dots,x_n)=\binom{m-\alpha}{m}\binom{N}{x_1,\dots,x_n},\qquad \binom{m-\alpha}{m}:=\frac{(1-\alpha)(2-\alpha)\cdots(m-\alpha)}{m!}\ \ (m\ge1),
\label{eq:base-measure}
\end{equation}
with $\binom{-\alpha}0:=1$, and $\binom{N}{x_1,\dots,x_n}=N!/(x_1!\cdots x_n!)$ the ordinary multinomial coefficient. Equivalently, for $\theta$ with $u_i=e^{\theta^i}$, $s=\sum_iu_i$, the resulting probability mass function $p_\theta(x)=h(x)e^{\theta\cdot x}/M_\alpha(\theta)$ is exactly that of $(X_1,\dots,X_n)$ constructed by: (i) drawing $M\ge0$ from the law $P(M=m)\propto\binom{m-\alpha}{m}s^{2m}$ (a negative-binomial-type law with parameter $r=1-\alpha$ and success probability depending on $s$), (ii) drawing the parity bit $B:=N-2M\in\{0,1\}$ independently from $\operatorname{Bernoulli}\!\big(\tfrac s{1+s}\big)$ (\emph{not} a fair coin: $P(B=1)=s/(1+s)$, $P(B=0)=1/(1+s)$), and setting $N=2M+B$, and (iii) drawing $(X_1,\dots,X_n)\mid N\sim\operatorname{Multinomial}\!\big(N;p_1,\dots,p_n\big)$ with $p_i=u_i/s$.
\end{theorem}
\begin{proof}
By Definition~\ref{def:family}, $M_\alpha(\theta)=e^{\alpha\Psi_+(\theta)+(1-\alpha)\Psi_-(\theta)}=e^{\Psi_+(\theta)}\big)^\alpha\big(e^{\Psi_-(\theta)}\big)^{1-\alpha}=(1+s)^\alpha\big[(1-s)^{-1}\big]^{1-\alpha}=(1+s)^\alpha(1-s)^{\alpha-1}$, and $(1+s)^\alpha(1-s)^{\alpha-1}=(1+s)\cdot(1+s)^{\alpha-1}(1-s)^{\alpha-1}=(1+s)(1-s^2)^{\alpha-1}$, proving \eqref{eq:Malpha-factored}.

By the generalised binomial series (valid for $|s|<1$, which holds on $\Omega$ since $|s|\le\sum u_i<1$),
\[
(1-s^2)^{\alpha-1}=(1-s^2)^{-(1-\alpha)}=\sum_{m=0}^\infty\binom{m-\alpha}{m}s^{2m}.
\]
Multiplying by $(1+s)$ gives
\[
M_\alpha(\theta)=\sum_{m=0}^\infty\binom{m-\alpha}{m}s^{2m}+\sum_{m=0}^\infty\binom{m-\alpha}{m}s^{2m+1}=\sum_{N=0}^\infty\binom{\lfloor N/2\rfloor-\alpha}{\lfloor N/2\rfloor}s^N,
\]
i.e.\ the coefficient of $s^N$ in $M_\alpha$ is exactly $\binom{m-\alpha}{m}$ with $m=\lfloor N/2\rfloor$, for both $N=2m$ and $N=2m+1$. By the ordinary multinomial theorem, $s^N=\big(\sum_iu_i\big)^N=\sum_{x_1+\cdots+x_n=N}\binom{N}{x_1,\dots,x_n}u_1^{x_1}\cdots u_n^{x_n}$. Substituting and collecting the coefficient of $u_1^{x_1}\cdots u_n^{x_n}$ (with $N=\sum x_i$) gives exactly \eqref{eq:base-measure}, since $M_\alpha(\theta)=\sum_xh(x)e^{\theta\cdot x}=\sum_xh(x)\prod_iu_i^{x_i}$ by definition of the exponential-family log-partition function.

For the probabilistic description: the two terms $\binom{m-\alpha}m s^{2m}$ and $\binom{m-\alpha}m s^{2m+1}$ contributing to $M_\alpha(\theta)=\sum_Nh_N s^N$ (where $h_N:=\binom{\lfloor N/2\rfloor-\alpha}{\lfloor N/2\rfloor}$, before splitting $s^N$ via the multinomial theorem into individual $x$'s) are in the ratio $1:s$ at each fixed $m$; consequently, conditionally on $M=m$, the two values $N=2m$ and $N=2m+1$ have probabilities in the ratio $1:s$, i.e.\ $P(B=1\mid M=m)=s/(1+s)$ and $P(B=0\mid M=m)=1/(1+s)$ for every $m$, independently of $m$ -- this is precisely the stated $\operatorname{Bernoulli}(s/(1+s))$ law for $B=N-2M$, independent of $M$. Finally, once $N$ is fixed, \eqref{eq:base-measure}'s factor $\binom N{x_1,\dots,x_n}$ is, by the multinomial theorem applied to $s^N=(\sum_iu_i)^N$, exactly the un-normalised weight of $u_1^{x_1}\cdots u_n^{x_n}$; dividing through by $s^N=(\sum_iu_i)^N$ converts this into the $\operatorname{Multinomial}(N;p_1,\dots,p_n)$ probability mass function with $p_i=u_i/s$, confirming (iii).
\end{proof}

\begin{proposition}[Non-negativity of the base measure]
\label{prop:nonneg}
For every $\alpha\in[0,1]$ and every $m\ge0$, $\binom{m-\alpha}{m}\ge0$; consequently $h(x_1,\dots,x_n)\ge0$ in \eqref{eq:base-measure} for every $\alpha\in[0,1]$ and every $(x_1,\dots,x_n)\in\mathbb Z_{\ge0}^n$, so $M_\alpha$ is honestly the log-partition function of an exponential family with a non-negative base measure on $\mathbb Z_{\ge0}^n$, for every $\alpha\in[0,1]$: regular (full-dimensional natural parameter, infinite support) for $0\le\alpha<1$, and a finite-support boundary family (support $\{0,e_1,\dots,e_n\}$, $N\in\{0,1\}$ only) at $\alpha=1$, where the natural parameter is nonetheless still full-dimensional.
\end{proposition}
\begin{proof}
For $m=0$, $\binom{-\alpha}0=1>0$. For $m\ge1$, $\binom{m-\alpha}m=\frac1{m!}\prod_{j=1}^m(j-\alpha)$; since $\alpha\in[0,1]$ and $j\ge1$, each factor satisfies $j-\alpha\ge j-1\ge0$, so the product, and hence $\binom{m-\alpha}m$, is non-negative. Since the ordinary multinomial coefficient $\binom N{x_1,\dots,x_n}$ is always non-negative, \eqref{eq:base-measure} is a product of two non-negative factors.
\end{proof}

\begin{remark}
The parameter $\alpha$ therefore has a second, purely probabilistic meaning, entirely independent of its role in Sections~2--\ref{sec:kasner-dynamics}: writing $r:=1-\alpha$, the law of $M$ in Theorem~\ref{thm:compound} is exactly $M\sim\operatorname{NB}(r,p=s^2)$ in the standard parametrisation $P(M=m)=\binom{m+r-1}m(1-p)^rp^m$ (matching $\binom{m-\alpha}m=\binom{m+r-1}m$ and $(1-s^2)^{1-\alpha}=(1-p)^r$). The endpoint $\alpha=1$ ($r=0$) is the boundary degeneration of the negative binomial at $r=0$, forcing $M\equiv0$ a.s., i.e.\ $N\in\{0,1\}$ only; the resulting two-point exponential family is exactly the $N\in\{0,1\}$ truncation of $\Psi_+$'s own family, not literally ``the multinomial law'' itself (which allows arbitrary $N$). The endpoint $\alpha=0$ ($r=1$) gives the geometric-type case, recovering $\Psi_-$. The case $n=2$ specialises \eqref{eq:base-measure} to $h(x_1,x_2)=\binom{m-\alpha}m\binom N{x_1}$, a compound negative-binomial--binomial law whose angular part, conditionally on $N$, is $\operatorname{Binomial}(N,p_1)$ with $p_1=u_1/s$ (uniform, $p_1=\tfrac12$, only when $u_1=u_2$).
\end{remark}

\section{The Canonical Divergence and a Generalised Pythagorean Theorem}
\label{sec:divergence}

We now consider the Legendre dual $\Psi_\alpha^*(\eta):=\sup_\theta\big[\theta\cdot\eta-\Psi_\alpha(\theta)\big]$ and the associated canonical (Bregman) divergence, and show that a genuine Pythagorean decomposition holds for every $\alpha\in[0,1]$, with a projection point that is, strikingly, completely independent of $\alpha$.

\begin{definition}[Canonical divergence]
\label{def:divergence}
For $\theta,\theta'\in\Omega$, write $\eta(\theta'):=\nabla\Psi_\alpha(\theta')$ and define
\begin{equation}
D_\alpha(\theta\,\|\,\theta'):=\Psi_\alpha(\theta)+\Psi_\alpha^*(\eta(\theta'))-\theta\cdot\eta(\theta').
\label{eq:divergence-def}
\end{equation}
\end{definition}

\begin{proposition}[Bregman form and non-negativity]
\label{prop:bregman}
$D_\alpha(\theta\|\theta')=\Psi_\alpha(\theta)-\Psi_\alpha(\theta')-\eta(\theta')\cdot(\theta-\theta')$, and this quantity is computable directly from $\Psi_\alpha$ alone, without requiring an explicit formula for $\Psi_\alpha^*$. Moreover $D_\alpha(\theta\|\theta')\ge0$, with equality iff $\theta=\theta'$.
\end{proposition}
\begin{proof}
By definition of the Legendre transform, $\Psi_\alpha^*(\eta(\theta'))=\theta'\cdot\eta(\theta')-\Psi_\alpha(\theta')$ (the supremum defining $\Psi_\alpha^*$ at $\eta=\eta(\theta')$ is attained exactly at $\theta=\theta'$, since $\eta(\theta')=\nabla\Psi_\alpha(\theta')$ is precisely the first-order condition for that supremum). Substituting into \eqref{eq:divergence-def} gives $D_\alpha(\theta\|\theta')=\Psi_\alpha(\theta)+\theta'\cdot\eta(\theta')-\Psi_\alpha(\theta')-\theta\cdot\eta(\theta')=\Psi_\alpha(\theta)-\Psi_\alpha(\theta')-\eta(\theta')\cdot(\theta-\theta')$, which involves only $\Psi_\alpha$ and its gradient. Non-negativity is the standard first-order characterisation of convexity: since $\Psi_\alpha$ is strictly convex (Theorem~\ref{thm:posdef} gives $\operatorname{Hess}\Psi_\alpha=g_\alpha\succ0$ everywhere on $\Omega$), the first-order Taylor remainder $\Psi_\alpha(\theta)-\Psi_\alpha(\theta')-\nabla\Psi_\alpha(\theta')\cdot(\theta-\theta')$ is $\ge0$ for all $\theta,\theta'\in\Omega$, with equality iff $\theta=\theta'$.
\end{proof}

Let $M:=\{\theta\in\Omega:\theta^1=\theta^2=\cdots=\theta^n\}$, the (one-dimensional) locus fixed by the permutation symmetry of $\Psi_\alpha$.

\begin{proposition}[$M$ is $e$-flat]
\label{prop:eflat}
$M$ is a totally geodesic submanifold for the $e$-connection of the dually flat structure $(\Omega,g_\alpha,\Psi_\alpha,\Psi_\alpha^*)$.
\end{proposition}
\begin{proof}
By the standard construction of a dually flat space from a strictly convex potential \cite{AmariNagaoka2000}, the natural parameters $\theta$ are, by definition, affine coordinates for the $e$-connection (its Christoffel symbols vanish identically in $\theta$-coordinates). A submanifold cut out by linear equations in an affine coordinate system is automatically totally geodesic for the corresponding connection; $M$ is cut out by the $n-1$ linear equations $\theta^i-\theta^1=0$ ($i=2,\dots,n$), hence is $e$-flat.
\end{proof}

\begin{theorem}[The projection onto $M$ is the arithmetic mean]
\label{thm:projection}
For $\theta\in\Omega$, the point $\theta^\ast\in M$ minimising $D_\alpha(\theta\|\theta')$ over $\theta'\in M$ (the standard \emph{$m$-projection} of $\theta$ onto the $e$-flat submanifold $M$, in the sense of \cite{AmariNagaoka2000}: $\theta$ is held fixed as the first argument throughout) is $\theta^{\ast,i}=\bar\theta:=\tfrac1n\sum_j\theta^j$ for every $i$, \emph{independently of $\alpha$}.
\end{theorem}
\begin{proof}
Parametrise $M$ by the single scalar $c$, $\theta'=(c,\dots,c)$. By the permutation symmetry of $\Psi_\alpha$, $\eta_i(c,\dots,c)=\eta(c)$ is the same scalar function for every $i$. Restricting $D_\alpha(\theta\|\cdot)$ to $M$,
\[
G(c):=D_\alpha\big(\theta\,\|\,(c,\dots,c)\big)=\Psi_\alpha(\theta)-\Psi_\alpha(c,\dots,c)-\eta(c)\sum_i(\theta^i-c).
\]
Differentiating, and using the chain rule identity $\tfrac{d}{dc}\Psi_\alpha(c,\dots,c)=\sum_i\partial_i\Psi_\alpha(c,\dots,c)=n\,\eta(c)$ (each of the $n$ partials equal to $\eta(c)$ by symmetry),
\[
G'(c)=-n\eta(c)-\eta'(c)\sum_i(\theta^i-c)+n\eta(c)=\eta'(c)\Big[nc-\sum_i\theta^i\Big].
\]
Since $\eta(c)=\partial_1\Psi_\alpha(c,\dots,c)$, differentiating once more and using the chain rule and the permutation symmetry (each $\partial_j\partial_1\Psi_\alpha(c,\dots,c)=g_{\alpha,1j}(c,\dots,c)$),
\[
\eta'(c)=\sum_j\partial_j\partial_1\Psi_\alpha(c,\dots,c)=\sum_jg_{\alpha,1j}(c,\dots,c)=\big[g_\alpha(c,\dots,c)\mathbf 1\big]_1=\frac1n\,\mathbf 1^{\mathsf T}g_\alpha(c,\dots,c)\mathbf 1,
\]
the last equality using that every row-sum of $g_\alpha(c,\dots,c)$ equals the same value $\eta'(c)$ by the permutation symmetry, so their sum, $\mathbf1^{\mathsf T}g_\alpha\mathbf1$, is $n$ times any one of them. Since $g_\alpha(c,\dots,c)\succ0$ by Theorem~\ref{thm:posdef} and $\mathbf1\ne0$, this quadratic form is strictly positive, so $\eta'(c)>0$ for every $c$; hence $G'(c)=0$ if and only if $c=\bar\theta$. At $c=\bar\theta$, $G''(\bar\theta)=\eta''(\bar\theta)\cdot0+n\eta'(\bar\theta)=n\eta'(\bar\theta)>0$, so $c=\bar\theta$ is the unique minimiser. This argument used only the permutation symmetry and strict convexity of $\Psi_\alpha$, neither of which depends on $\alpha$; hence $\theta^\ast=(\bar\theta,\dots,\bar\theta)$ for every $\alpha\in[0,1]$.
\end{proof}

\begin{theorem}[Generalised Pythagorean theorem]
\label{thm:pythagoras}
For every $\theta\in\Omega$, every $\theta'\in M$, and every $\alpha\in[0,1]$,
\begin{equation}
D_\alpha(\theta\,\|\,\theta')=D_\alpha(\theta\,\|\,\theta^\ast)+D_\alpha(\theta^\ast\,\|\,\theta'),\qquad \theta^\ast=(\bar\theta,\dots,\bar\theta).
\label{eq:pythagoras}
\end{equation}
\end{theorem}
\begin{proof}
Write out the right-hand side using Proposition~\ref{prop:bregman}:
\[
D_\alpha(\theta\|\theta^\ast)+D_\alpha(\theta^\ast\|\theta')=\big[\Psi_\alpha(\theta)-\Psi_\alpha(\theta^\ast)-\eta(\theta^\ast)\cdot(\theta-\theta^\ast)\big]+\big[\Psi_\alpha(\theta^\ast)-\Psi_\alpha(\theta')-\eta(\theta')\cdot(\theta^\ast-\theta')\big].
\]
The $\Psi_\alpha(\theta^\ast)$ terms cancel, leaving $\Psi_\alpha(\theta)-\Psi_\alpha(\theta')-\eta(\theta^\ast)\cdot(\theta-\theta^\ast)-\eta(\theta')\cdot(\theta^\ast-\theta')$. Comparing with $D_\alpha(\theta\|\theta')=\Psi_\alpha(\theta)-\Psi_\alpha(\theta')-\eta(\theta')\cdot(\theta-\theta')$, the identity \eqref{eq:pythagoras} is equivalent to
\[
\eta(\theta^\ast)\cdot(\theta-\theta^\ast)+\eta(\theta')\cdot(\theta^\ast-\theta')=\eta(\theta')\cdot(\theta-\theta'),
\quad\text{i.e.}\quad
\big[\eta(\theta^\ast)-\eta(\theta')\big]\cdot(\theta-\theta^\ast)=0.
\]
Since $\theta',\theta^\ast\in M$, both have all coordinates equal (to $c'$ and $\bar\theta$ respectively), so by the permutation symmetry used in Theorem~\ref{thm:projection}, $\eta(\theta^\ast)-\eta(\theta')=\big[\eta(\bar\theta)-\eta(c')\big]\cdot(1,\dots,1)$, a vector proportional to $(1,\dots,1)$. Hence
\[
\big[\eta(\theta^\ast)-\eta(\theta')\big]\cdot(\theta-\theta^\ast)=\big[\eta(\bar\theta)-\eta(c')\big]\sum_i(\theta^i-\bar\theta)=\big[\eta(\bar\theta)-\eta(c')\big]\cdot\Big(\sum_i\theta^i-n\bar\theta\Big)=0,
\]
using $\sum_i\theta^i=n\bar\theta$ by definition of $\bar\theta$. This holds for every $\theta'\in M$, proving \eqref{eq:pythagoras}.
\end{proof}

\begin{remark}[A sharper robustness statement]
\label{rem:robustness}
Theorem~\ref{thm:projection} strengthens the contrast noted for the master equation \eqref{eq:master}: whereas the microscopic Kasner dynamics for $\theta^1(t),\theta^2(t)$ loses all polynomial structure under mixing (Section~5's remarks on the general-$\kappa$ system), becoming a high-degree rational system for $0<\alpha<1$, the Pythagorean projection point $\theta^\ast$ is not merely computable in closed form for every $\alpha$ -- it is \emph{exactly the same point, for every $\alpha\in[0,1]$}, depending only on the permutation symmetry shared by $h_+$ and $h_-$ and not on the specific convex combination. What does depend on $\alpha$ is only the \emph{value} $D_\alpha(\theta\|\theta^\ast)$ of the resulting ``anisotropy'' divergence at that fixed point, computable in closed form from Proposition~\ref{prop:bregman} using only $\Psi_\alpha$ and $\eta$, again without requiring the (generally unavailable, cf.\ Remark~\ref{rem:not-mixture}) closed form of $\Psi_\alpha^\ast$ itself.
\end{remark}

\section{Geodesic Triangles and the Gauss--Bonnet Theorem}
\label{sec:gaussbonnet}

The generalised Pythagorean theorem of Section~\ref{sec:divergence} uses the dual ($e$/$m$-connection) notion of orthogonality, which is a different geometric structure from the ordinary Levi-Civita connection of $(\Omega,g_\alpha)$ whenever $g_\alpha$ is not flat -- in particular at the Einstein-space endpoints $\alpha=0,1$, where the curvature is constant but non-zero. We show that $M$ is nonetheless totally geodesic in the ordinary Riemannian sense, that the dual right angle at $\theta^\ast$ is \emph{not}, in general, a Levi-Civita right angle, and that the classical Gauss--Bonnet theorem correctly predicts the sign of the angle-sum excess of an honest geodesic triangle at the two Einstein-space endpoints -- with no closed-form analogue for the critical mixing parameter at which the excess vanishes.

\begin{proposition}[$M$ is totally geodesic]
\label{prop:M-totally-geodesic}
For every $\alpha\in[0,1]$, $M=\{\theta^1=\cdots=\theta^n\}$ is a totally geodesic submanifold of $(\Omega,g_\alpha)$ for the Levi-Civita connection.
\end{proposition}
\begin{proof}
The symmetric group $S_n$ acts on $\Omega$ by permuting the coordinates $\theta^1,\dots,\theta^n$. Since $u_i=e^{\theta^i}$ and both $\Psi_+(\theta)=\log(1+\sum_iu_i)$ and $\Psi_-(\theta)=-\log(1-\sum_iu_i)$ depend on $\theta$ only through the unordered set $\{u_1,\dots,u_n\}$, each is invariant under this action, hence so is $\Psi_\alpha=\alpha\Psi_++(1-\alpha)\Psi_-$ for every $\alpha$; consequently every $\sigma\in S_n$ pulls $g_\alpha=\operatorname{Hess}\Psi_\alpha$ back to itself and so acts as an isometry of $(\Omega,g_\alpha)$. The fixed-point set of this isometric group action is exactly $M$. By the standard fact that the fixed-point set of a group of isometries of a Riemannian manifold is totally geodesic (see e.g.\ \cite{KobayashiNomizu}), $M$ is totally geodesic for every $\alpha\in[0,1]$.
\end{proof}

\begin{proposition}[The dual right angle is not, in general, a Levi-Civita right angle]
\label{prop:not-riemannian-right-angle}
There exist $\theta\in\Omega$ and $\alpha\in\{0,1\}$ for which the Levi-Civita geodesic joining $\theta$ to its Pythagorean projection $\theta^\ast$ (Theorem~\ref{thm:projection}) does not meet $M$ orthogonally.
\end{proposition}
\begin{proof}[Numerical verification]
Take $n=2$, $\alpha=1$, $\theta=(\log0.05,\log0.2)$, so $\theta^\ast=(\bar\theta,\bar\theta)$ with $\bar\theta=\tfrac12(\log0.05+\log0.2)$. Solving the geodesic equation $\ddot\theta^k+\Gamma^k_{ij}(\theta)\dot\theta^i\dot\theta^j=0$ for $h_+=g_1$ by a shooting method (matching the endpoint $\theta$ from initial point $\theta^\ast$ to machine precision) and comparing the resulting initial velocity to the tangent direction $(1,1,\dots,1)$ of $M$ at $\theta^\ast$, using the inner product induced by $g_1(\theta^\ast)$, gives $\cos(\angle)\approx0.1547$, i.e.\ an angle of approximately $81.1^\circ$, not $90^\circ$. The dual orthogonality condition established in the proof of Theorem~\ref{thm:pythagoras} therefore does not coincide with Levi-Civita orthogonality here; this is expected, since the $e$- and $m$-connections agree with the Levi-Civita connection only when the space is flat, which by Theorem~\ref{thm:endpoints} is not the case at $\alpha=0,1$ (curvature $\mp\tfrac14\ne0$).
\end{proof}

\begin{theorem}[Sign of the angle-sum excess]
\label{prop:gaussbonnet}
Let $\theta,\theta'\in\Omega$ with $\theta'\in M$, $\theta\notin M$, and let $\theta^\ast\in M$ be as in Theorem~\ref{thm:projection}. Suppose the three vertices $\theta,\theta^\ast,\theta'$ are joined by geodesic segments (realised by genuine $g_\alpha$-geodesics; the side $\theta^\ast\theta'$ lies in $M$ by Proposition~\ref{prop:M-totally-geodesic}) forming a \emph{simple, non-degenerate} geodesic triangle -- i.e.\ the three segments meet only at the stated vertices and together bound a single embedded disk $\triangle\subset\Omega$ (at $\alpha=1$, where geodesics of $g_1$ can wrap around the sphere, this selects the shorter arc realising each side, so that $\triangle$ is unambiguous). Let $\Sigma(\alpha)$ denote the sum of the three interior angles of $\triangle$, measured with respect to $g_\alpha$. Then $\Sigma(1)>\pi$ and $\Sigma(0)<\pi$, for \emph{every} such triangle.
\end{theorem}
\begin{proof}
The Gauss--Bonnet theorem for a simple geodesic triangle $\triangle$ bounding an embedded disk, on a surface of constant curvature $K$, states
\[
\Sigma-\pi=K\cdot\operatorname{Area}(\triangle).
\]
This applies verbatim at $\alpha=0,1$, since $K_G(\theta,\alpha)\equiv\mp\tfrac14$ is constant there (Theorem~\ref{thm:endpoints}), independently of which such triangle is chosen. Since $\triangle$ is a non-degenerate embedded disk, $\operatorname{Area}(\triangle)>0$, so the sign of $\Sigma-\pi$ is exactly the sign of $K$, giving $\Sigma(1)>\pi$ (as $K=+\tfrac14>0$) and $\Sigma(0)<\pi$ (as $K=-\tfrac14<0$), for every simple, non-degenerate geodesic triangle with $\theta,\theta^\ast,\theta'$ as vertices, $\theta'\in M$.
\end{proof}

\begin{proof}[Numerical illustration]
For the specific triangle $\theta=(\log0.05,\log0.2)$, $\theta'=(-3,-3)$, solving the three geodesic boundary-value problems by shooting and summing the resulting interior angles (computed via the $g_\alpha$-inner product of the relevant initial velocities, as in Proposition~\ref{prop:not-riemannian-right-angle}) gives $\Sigma(1)\approx180.469^\circ$ and $\Sigma(0)\approx179.201^\circ$, concretely confirming the theorem for this triangle.
\end{proof}

\begin{figure}[h]
\centering
\includegraphics[width=0.95\textwidth]{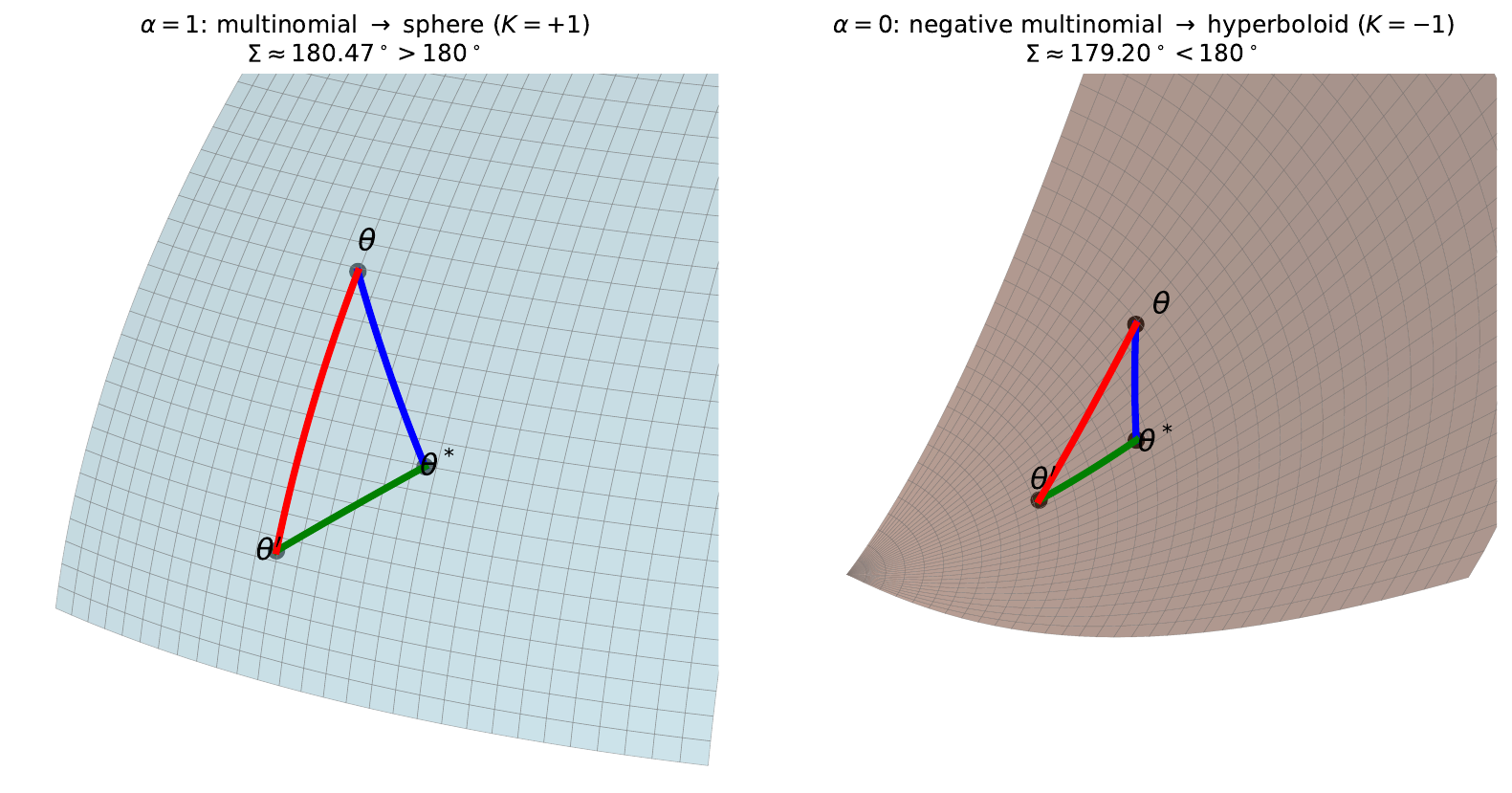}
\caption{The geodesic triangle $\theta,\theta^\ast,\theta'$ of Proposition~\ref{prop:gaussbonnet}, drawn on the two isometric embeddings identified in Theorem~\ref{thm:endpoints}: the multinomial sphere at $\alpha=1$ (left; the triangle's sides bulge outward, giving angle sum $\Sigma>180^\circ$) and the negative-multinomial hyperboloid at $\alpha=0$ (right; the sides bow inward relative to the saddle-shaped surface, giving $\Sigma<180^\circ$). Both panels use the same $\theta=(\log0.05,\log0.2)$, $\theta^\ast=(\bar\theta,\bar\theta)$, $\theta'=(-3,-3)$; the sphere embedding is $(\sqrt{p_0},\sqrt{p_1},\sqrt{p_2})$ and the hyperboloid embedding is $(\sqrt{1+\eta_1+\eta_2},\sqrt{\eta_1},\sqrt{\eta_2})$, both exact isometries up to the overall curvature normalisation $K=\pm1$ versus $K=\pm\tfrac14$.}
\label{fig:geodesic-triangles}
\end{figure}

\begin{remark}[Existence of a critical mixing parameter; no closed form obtained; no universality across triangles]
\label{rem:no-universal-crossing}
By continuity of $\alpha\mapsto\Sigma(\alpha)$ (a consequence of the continuous dependence of solutions of the geodesic equation on parameters) and Proposition~\ref{prop:gaussbonnet}, the intermediate value theorem guarantees, for every triangle as above, \emph{existence} of at least one $\alpha^\dagger=\alpha^\dagger(\theta,\theta')\in(0,1)$ with $\Sigma(\alpha^\dagger)=\pi$. We have not proved \emph{uniqueness} of this $\alpha^\dagger$ in general; for the specific triangle $\theta=(\log0.05,\log0.2)$, $\theta'=(-3,-3)$, a numerical scan over a fine grid of $20$ values of $\alpha\in(0,1)$ showed $\Sigma(\alpha)-\pi$ strictly increasing throughout, consistent with (but not a proof of) uniqueness there; establishing uniqueness in general would require showing $d\Sigma/d\alpha>0$ analytically, which we leave open. Unlike the pointwise critical mixing parameter $\alpha^\ast(w)$ of Theorem~\ref{thm:criticalalpha}, which we exhibited as a single closed-form algebraic function of the position $w$ alone, we do not obtain a closed-form expression for $\alpha^\dagger$: since $K_G(\theta,\alpha)$ is non-constant in $\theta$ for $0<\alpha<1$ (Theorem~\ref{thm:nonconstant}), $\Sigma(\alpha)-\pi=\iint_\triangle K_G(\theta,\alpha)\,dA$ is an area integral of a non-constant integrand over a geodesically-bounded region whose boundary is itself known only via numerical solution of the geodesic equation, and we did not find a way to evaluate it in elementary terms; this is a statement about what we were able to compute, not a proof that no closed form exists. Numerically, for three distinct triangles (fixing $\theta'=(c,c)$ and varying $\theta$),
\begin{align*}
(\theta,\theta')&=\big((\log0.05,\log0.2),\,(-3,-3)\big):&\alpha^\dagger&\approx0.8005,\\
(\theta,\theta')&=\big((\log0.02,\log0.35),\,(-4,-4)\big):&\alpha^\dagger&\approx0.7567,\\
(\theta,\theta')&=\big((\log0.15,\log0.1),\,(-1.5,-1.5)\big):&\alpha^\dagger&\approx0.9187,
\end{align*}
showing that $\alpha^\dagger$ (or, more precisely, at least one such critical value for each triangle) genuinely varies with the triangle. This sharpens the moral of Theorem~\ref{thm:criticalalpha} and Section~\ref{sec:divergence}: every notion of ``the special value of $\alpha$'' produced by the curved geometry of $g_\alpha$ -- whether a pointwise curvature zero or a Gauss--Bonnet angle-sum zero -- is a genuine \emph{function} of auxiliary data (a position $w$, or a triangle $(\theta,\theta')$) rather than a universal constant, in contrast to the $\alpha$-independent invariants identified elsewhere in this paper: positive-definiteness of $g_\alpha$ (Theorem~\ref{thm:posdef}), the resulting constraint $\beta\le0$ on the Kasner discriminant (Theorem~\ref{thm:dictionary}(a)), and the Pythagorean projection point $\theta^\ast$ (Theorem~\ref{thm:projection}).
\end{remark}

\section{Discussion}

The family $g_\alpha$ furnishes a fully computable example in which several ordinarily-conflated notions of ``hyperbolicity'' turn out to be logically independent. The intrinsic Gaussian curvature of $(\Omega,g_\alpha)$ (Sections~2--3) is $\pm\tfrac14$ only at the two Einstein-space endpoints, changing sign at the explicit, position-dependent locus $\alpha^\ast(w)$ of Theorem~\ref{thm:criticalalpha} in between. The Kasner volume equation built from $g_\alpha$ (Section~\ref{sec:kasner-dynamics}) is governed by a genuinely different pair of invariants: positive-definiteness of $g_\alpha$, uniform in $\alpha$, forces the conserved quantity $\beta\le0$, but the qualitative branch of the equation is in fact decided by a second conserved quantity $C=4\det K(0)$, and this sign is not determined by positive-definiteness of the initial metric alone -- both signs of $C$ arise as \emph{initial Kasner data} from a single, fixed positive-definite $g_\alpha(\theta_0)$, depending on the initial velocity (Theorem~\ref{thm:dictionary}(c)); we verified directly that neither sign is realised by a full trajectory remaining on $\{g_\alpha(\theta):\theta\in\Omega\}$ for $t\ne0$, since the master equation restricted to this $2$-dimensional submanifold of the $3$-dimensional symmetric-matrix space is already inconsistent at $t=0$ (Remark~\ref{rem:scope-of-c}) -- so the discriminant dictionary is properly a theorem about the abstract matrix equation, to which $g_\alpha$ supplies initial data only. The explicit solution machinery of Section~5, in particular the matrix-exponential formula \eqref{eq:fullsolution}, shows this richer structure is nonetheless fully solvable in closed form throughout. Section~\ref{sec:divergence} exhibits a further, remarkably robust invariant: the Pythagorean decomposition point $\theta^\ast$, fixed throughout the entire family by symmetry alone, independently of $\alpha$. Section~\ref{sec:gaussbonnet} completes the picture by confirming, via the classical Gauss--Bonnet theorem, that the sign of the intrinsic curvature is faithfully detected by ordinary geodesic triangles at the two Einstein-space endpoints, while the resulting critical mixing parameter -- unlike $\alpha^\ast(w)$ -- depends on the whole triangle, not merely on a position. Altogether, four logically distinct notions of ``sign'' coexist in this one family: the intrinsic curvature's sign, the metric's signature (always definite here, forcing $\beta\le0$), the Kasner discriminant $C$'s sign (genuinely free), and the Gauss--Bonnet angle-sum excess -- a concrete illustration of how many distinct invariants a single convex-combination construction can carry, and how carefully they must be distinguished. In eigenvalue terms (writing $k_1,k_2$ for the eigenvalues of $K$), $\beta=-\Delta(k_1-k_2)^2/2$ measures their \emph{separation} while $C=4k_1k_2$ measures their \emph{product}; positive-definiteness of $h$ constrains only the former to have a sign, leaving the latter free.

We close by listing the open questions this work leaves explicit. (i) Whether \emph{any} nontrivial curve $\theta(t)\in\Omega$ keeps $g_\alpha(\theta(t))$ an exact solution of \eqref{eq:master} beyond an instant -- Remark~\ref{rem:scope-of-c} shows the two initial velocities tested there fail, but the question for general initial data on $\{g_\alpha(\theta)\}$ remains open. (ii) A closed-form expression for the Gauss--Bonnet critical parameter $\alpha^\dagger(\theta,\theta')$ of Remark~\ref{rem:no-universal-crossing}, or a proof that none exists. (iii) Uniqueness of $\alpha^\dagger$ for a general triangle (we verified monotonicity of $\Sigma(\alpha)$ only for one example). (iv) The general-$n$ analogue of the symmetric-locus curvature formula of Theorem~\ref{thm:symcurve}, obtained here only for $n=2$.

\end{document}